\documentclass[3p,times]{elsarticle}

\journal{Elsevier}

\usepackage{amsmath}
\usepackage{mathrsfs} 
\usepackage{url}
\newcommand{\domD}{\mathscr{D}}
\renewcommand{\pi}{\piup}

\renewcommand{\Im}{\operatorname{Im}}

\DeclareMathOperator{\Si}{Si}
\DeclareMathOperator{\OO}{O}
\DeclareMathOperator{\E}{e}
\DeclareMathOperator{\I}{i}
\DeclareMathOperator{\Beta}{B}
\newcommand{\D}{\,\mathrm{d}}
\newcommand{\mathbd}{\boldsymbol}
\newcommand{\LC}{\mathbf{L}}
\newcommand{\HC}{\mathbf{HC}}
\newcommand{\MC}{\mathbf{M}}
\newdefinition{definition}{Definition}[section]
\newtheorem{theorem}{Theorem}[section]
\newtheorem{lemma}[theorem]{Lemma}
\newtheorem{proposition}[theorem]{Proposition}
\newtheorem{corollary}[theorem]{Corollary}
\newtheorem{remark}{Remark}[section]
\newtheorem{example}{Example}[section]
\newproof{proof}{Proof}
\numberwithin{equation}{section}
\newcommand{\textSE}{\text{\tiny{\rm{SE}}}}
\newcommand{\textDE}{\text{\tiny{\rm{DE}}}}
\newcommand{\Vol}{\mathcal{V}}
\newcommand{\VolSEn}{\mathcal{V}_N^{\textSE}}
\newcommand{\VolDEn}{\mathcal{V}_N^{\textDE}}
\newcommand{\Fred}{\mathcal{K}}
\newcommand{\FredSEn}{\mathcal{K}_N^{\textSE}}
\newcommand{\FredDEn}{\mathcal{K}_N^{\textDE}}
\newcommand{\SEt}{\psi^{\textSE}}
\newcommand{\DEt}{\psi^{\textDE}}
\newcommand{\SEtInv}{\{\SEt\}^{-1}}
\newcommand{\DEtInv}{\{\DEt\}^{-1}}
\newcommand{\SEtDiv}{\{\SEt\}'}
\newcommand{\DEtDiv}{\{\DEt\}'}
\newcommand{\ProjSE}{\mathcal{P}_N^{\textSE}}
\newcommand{\ProjDE}{\mathcal{P}_N^{\textDE}}
\newcommand{\vSEn}{v_N^{\textSE}}
\newcommand{\vDEn}{v_N^{\textDE}}
\newcommand{\uSEn}{u_N^{\textSE}}
\newcommand{\uDEn}{u_N^{\textDE}}

\begin{document}

\begin{frontmatter}

\title{Improvement of Sinc-collocation methods for Volterra-Fredholm integral
 equations of the second kind and their theoretical analysis~\tnoteref{mytitlenote}}
\tnotetext[mytitlenote]{This study was partially supported by JSPS
Grant-in-Aid for Scientific Research (C) JP23K03218.}

\author{Tomoaki Okayama}
\address{Graduate School of Information Sciences, Hiroshima City University,
3-4-1, Ozuka-higashi, Asaminami-ku, Hiroshima 731-3194, Japan}
\ead{okayama@hiroshima-cu.ac.jp}




\begin{abstract}
Sinc-collocation methods for Volterra-Fredholm integral equations
of the second kind were proposed independently by multiple authors:
by Shamloo et al.\ in 2012 and by Mesgarani and Mollapourasl in 2013.
Their theoretical analyses and numerical experiments suggest that
the presented methods can attain root-exponential convergence.
However, their convergence has not been strictly proved.
This study improves these methods to facilitate implementation,
and provides a convergence theorem for the improved method.
For the same equations, another Sinc-collocation method was
proposed in 2016 by John and Ogbonna,
which is regarded as an improvement to
the variable transformation employed by Shamloo et al.
It may attain a higher rate than
the previous methods, but its convergence has not yet been proved.
Therefore,
this study improves it to facilitate implementation,
and provides a convergence theorem for the improved method.
\end{abstract}


\end{frontmatter}

\section{Introduction and summary}
\label{sec:intro}

This study focuses on Sinc-collocation methods
for Volterra-Fredholm integral equations of the second kind of the form
\begin{equation}
 u(t)
 - \int_a^t k_1(t, s)u(s)\mathrm{d}s
 - \int_a^b k_2(t, s)u(s)\mathrm{d}s
= g(t),\quad a\leq t\leq b,
\label{eq:VFIE}
\end{equation}
where $k_1$, $k_2$, and $g$ are given continuous functions,
and $u$ is the solution to be determined.
This type of equations arises in application of various
fields~\cite{mirzaee16:_numer_volter_fredh}, and for this reason,
several numerical methods for the equations have been
developed~\cite{calio10:_direc,fermo25:_volter_fredh,negarchi18:_numer_volter_fredh,nemati15:_numer_volter_fredh,wang13:_lagran_volter_fredh}.
Most numerical methods in the literature
assume that the given functions are continuously differentiable
on the whole interval $[a, b]$,
and consequently such methods exhibit reduced efficiency
when any one of given functions has derivative singularity at the endpoints.
To address this problem,
Shamloo et al.~\cite{Shamloo} proposed a Sinc-collocation method,
which does not require differentiability at the endpoints.
Their method was derived by a typical collocation procedure
based on the Sinc approximation
\begin{equation}
\label{eq:Sinc-approximation}
 F(x) \approx \sum_{j=-N}^N F(jh)S(j,h)(x),\quad x\in\mathbb{R},
\end{equation}
where $h$ is a mesh size selected appropriately depending on
a positive integer $N$, and
$S(j,h)$ is the so-called Sinc function defined by
\[
 S(j,h)(x)=
\begin{cases}
\dfrac{\sin[\pi(x - jh)/h]}{\pi(x - jh)/h} & (x\neq jh), \\
1 & (x = jh).
\end{cases}
\]
The Sinc approximation is known as a highly efficient
formula~\cite{stenger93:_numer},
but two issues need to be addressed.
The first issue is that the given interval $(a, b)$ in~\eqref{eq:VFIE}
is finite, whereas the variable $x$ in~\eqref{eq:Sinc-approximation}
should be able to move over the entire real axis $\mathbb{R}$.
A standard way to resolve this issue is
to employ the tanh transformation
\[
  t = \SEt(x) = \frac{b - a}{2}\tanh\left(\frac{x}{2}\right)+\frac{b+a}{2},
\]
which maps $\mathbb{R}$ onto $(a, b)$.
In this paper, this is referred to as
the single-exponential (SE) transformation.
Combining the Sinc approximation with the SE transformation,
we obtain the following formula
\begin{equation}
\label{eq:SE-Sinc-approximation}
 f(t) \approx \sum_{j=-N}^N f(\SEt(jh))S(j,h)(\SEtInv(t)),\quad
t\in (a,b).
\end{equation}
The second issue is that the right-hand side of~\eqref{eq:SE-Sinc-approximation}
tends to $0$ as $t\to a$ or $t\to b$.
This implies that the right-hand side cannot give an accurate approximation
near the endpoints if $f(a)\neq 0$ or $f(b)\neq 0$.
A standard way to resolve this issue is to introduce auxiliary basis functions
\[
  \omega_a(t) = \frac{b - t}{b - a},\quad \omega_b(t)=\frac{t - a}{b - a}.
\]
Based on the two standard ways,
Shamloo et al.~\cite{Shamloo} set their approximate solution $\uSEn$ as
\begin{equation}
 \uSEn(t) = c_{-N}\omega_a(t)
+\sum_{j=-N+1}^{N-1} c_j S(j,h)(\SEtInv(t)) + c_{N} \omega_b(t),
\label{eq:uSEn}
\end{equation}
and applied a usual collocation procedure, setting
collocation points at $t = \SEt(ih)$ ($i=-N,\,\ldots,\,N$)
to obtain the unknown coefficients $c_{-N},\,\ldots,\,c_{N}$.
They conducted an error analysis of the method and showed that
there exists a constant $C$ independent of $N$ such that
\begin{equation}
 \sup_{t\in(a,b)}|u(t) - \uSEn(t)|
\leq C \mu_N \sqrt{N} \E^{-\sqrt{\pi d \alpha N}},
\label{eq:Shamloo-error-analysis}
\end{equation}
where $d$ and $\alpha$ are positive constants,
and $\mu_N$ denotes the 2-norm of the inverse of the coefficient matrix of
the resulting system.
The error estimate and their numerical examples
suggest that their method may attain root-exponential convergence.

However, three points remain to be discussed concerning their results.
The first point is that $u(a)=u(b)=0$ is assumed to derive the
error estimate, which is not generally satisfied.
The second point is that even if $u(a)=u(b)=0$ is satisfied,
its convergence is not rigorously proved
because $\mu_N$ is not estimated with respect to $N$.
The third point is that the implementation of their method is rather
complicated. This is mainly because
$\uSEn(\SEt(ih)) =  c_i$ does not hold.
This defect was taken over in extension
to nonlinear equations~\cite{zarebnia13:_volter}.

Another Sinc-collocation method was proposed by
Mesgarani and Mollapourasl~\cite{MesMoll}. Their approximate
solution $\tilde{u}_N^{\textSE}$ is simply set
based on the formula~\eqref{eq:SE-Sinc-approximation} as
\[
\tilde{u}_N^{\textSE}(t)
 = \sum_{j=-N}^{N} c_j S(j,h)(\SEtInv(t)).
\]
Their resulting system of linear equations is simple because
$\tilde{u}_N^{\textSE}(\SEt(ih)) = c_i$ \emph{does} hold.
However, as pointed out for~\eqref{eq:SE-Sinc-approximation},
their solution generally produces poor results near the endpoints
except when $u(a)=u(b)=0$.
Therefore, although they conducted an error analysis on their method,
their result is difficult to justify.

To address the aforementioned challenges,
this study proposes an easy-to-implement Sinc-collocation method,
where the approximate solution $\vSEn$ is set as
\begin{equation}
 \vSEn(t)
= c_{-N}\omega_a(t)
 + \sum_{j=-N}^N\left\{c_j - c_{-N}\omega_a(\SEt(jh)) - c_N\omega_b(\SEt(jh))\right\}S(j,h)(\SEtInv(t))
+ c_N\omega_b(t),
\label{eq:vSEn}
\end{equation}
which satisfies $\vSEn(\SEt(ih)) = c_i$.
This concept has already been utilized for several
equations~\cite{okayama18:_theor_sinc,OkayamaFredholm,OkayamaVolterra,stenger93:_numer},
and may work well for the equation~\eqref{eq:VFIE}.
In fact, this study theoretically shows that
the error of $\vSEn$ is bounded as
\[
 \max_{t\in [a,b]}|u(t) - \vSEn(t)|
\leq C \sqrt{N} \E^{-\sqrt{\pi d \alpha N}},
\]
which rigorously proves root-exponential convergence of $\vSEn$.

Furthermore, this study extends this result to
the Sinc-collocation method proposed by
John and Ogbonna~\cite{JohnOgbo}.
Their method can be approximately derived
from the Sinc-collocation method of Shamloo et al.~\cite{Shamloo}
by replacing the tanh transformation with the
so-called double-exponential (DE) transformation
\[
 t = \DEt(x) = \frac{b - a}{2}\tanh\left(\frac{\pi}{2}\sinh x\right)
+\frac{b+a}{2}.
\]
During recent decades, researchers have indicated that
such replacement generally enhances the convergence rate
of the numerical methods based on the Sinc
approximation~\cite{Mori,SugiharaMatsuo},
which may have motivated the authors to employ the DE transformation
in the present case.
However, their variable transformation is not exactly
the original DE transformation,
and an inconsistency exists in the collocation points
(the details are discussed later).
By appropriately modifying their method,
this study derives a further improved Sinc-collocation method
attaining almost exponential convergence;
the error of the approximate solution $\vDEn$ is bounded as
\[
 \max_{t\in [a,b]}|u(t) - \vDEn(t)|
\leq C \E^{-\pi d N/\log(2 d N/\alpha)},
\]
which is proved theoretically.
The convergence rate is significantly higher than that of $\vSEn$.

The remainder of this paper is organized as follows.
As a preliminary, in Sect.~\ref{sec:preliminary},
we describe basic approximation formulas for
functions, definite integrals, and indefinite integrals,
along with their convergence theorems.
Existing Sinc-collocation methods by Shamloo et al.~\cite{Shamloo}
and John--Ogbonna~\cite{JohnOgbo} are
described in Sect.~\ref{sec:existing}.
New Sinc-collocation methods by this study and their
convergence theorems are described in Sect.~\ref{sec:new}.
Numerical examples are presented in Sect.~\ref{sec:numer}.
The proofs of the new theorems are provided in Sect.~\ref{sec:proofs}.


\section{Sinc approximation and Sinc indefinite integration}
\label{sec:preliminary}

In this section, approximation formulas
for functions, definite integrals, and indefinite integrals, based on
the Sinc approximation~\eqref{eq:Sinc-approximation} are described.
The associated convergence theorems are also stated.

\subsection{Generalized SE-Sinc and DE-Sinc approximations}

As explained in Sect.~\ref{sec:intro},
a naive combination of the Sinc approximation and
the SE transformation~\eqref{eq:SE-Sinc-approximation}
does not give an accurate approximation near the endpoints
if $f(a)\neq 0$ or $f(b)\neq 0$.
For a general boundary condition,
we introduce the approximation operator $\ProjSE: C([a, b])\to C([a, b])$
as
\begin{align}
 \ProjSE[f](t)
&= f(\SEt(-Nh))\omega_a(t) + f(\SEt(Nh))\omega_b(t) \nonumber\\
&\quad
 + \sum_{j=-N}^N \left\{
 f(\SEt(jh)) - f(\SEt(-Nh))\omega_a(\SEt(jh)) - f(\SEt(Nh))\omega_b(\SEt(jh))
\right\}S(j,h)(\SEtInv(t)).
\label{eq:ProjSE}
\end{align}
The approximation $f \approx \ProjSE f$ is referred to as
the generalized SE-Sinc approximation in this paper.
Similarly, replacing the SE transformation with the DE transformation,
we introduce the approximation operator $\ProjDE: C([a, b])\to C([a, b])$
as
\begin{align}
 \ProjDE[f](t)
&= f(\DEt(-Nh))\omega_a(t) + f(\DEt(Nh))\omega_b(t) \nonumber\\
&\quad
 + \sum_{j=-N}^N \left\{
 f(\DEt(jh)) - f(\DEt(-Nh))\omega_a(\DEt(jh)) - f(\DEt(Nh))\omega_b(\DEt(jh))
\right\}S(j,h)(\DEtInv(t)).
\label{eq:ProjDE}
\end{align}
The approximation $f \approx \ProjDE f$ is referred to as
the generalized DE-Sinc approximation in this paper.
The following function spaces are important in stating the
convergence theorems of the two approximations.

\begin{definition}
Let $\domD$ be a bounded and simply-connected complex domain
(or Riemann surface).
Then, $\HC(\domD)$ denotes the family of all functions $f$
that are analytic in $\domD$ and continuous on $\overline{\domD}$.
This function space is complete with the norm $\|\cdot\|_{\HC(\domD)}$
defined by $\|f\|_{\HC(\domD)} = \max_{z\in\overline{\domD}}|f(z)|$.
\end{definition}

\begin{definition}
Let $\alpha$ be a positive constant with $\alpha\leq 1$,
and let $\domD$ be a simply-connected complex domain
(or Riemann surface) that satisfies $(a, b)\subset \domD$.
Then, $\MC_{\alpha}(\domD)$ denotes the family of all functions
$f\in\HC(\domD)$ for which there exists a constant $C$ such that
for all $z\in\domD$,
\begin{align*}
|f(z) - f(a)|&\leq C \left|z - a\right|^{\alpha},\\
|f(b) - f(z)|&\leq C \left|b - z\right|^{\alpha}.
\end{align*}
\end{definition}

Here, $\domD$ is supposed to be a domain translated by $\SEt$ or $\DEt$
from the strip complex domain
\[
 \domD_d = \left\{\zeta\in\mathbb{C}: |\Im\zeta| < d\right\}
\]
for a positive constant $d$. To be more precise, we suppose either
\[
 \SEt(\domD_d) =
 \left\{z = \SEt(\zeta) : \zeta\in\domD_d\right\}
\]
or
\[
  \DEt(\domD_d) =
 \left\{z = \DEt(\zeta) : \zeta\in\domD_d\right\}.
\]
For functions belonging to $\MC_{\alpha}(\domD)$,
convergence theorems of the generalized SE/DE-Sinc approximations
were provided as follows.
Here, $\|\cdot\|_{C([a,b])}$ denotes the usual uniform norm over $[a, b]$.

\begin{theorem}[Okayama~{\cite[Theorem~3]{okayama13:_note}}]
\label{thm:SE-Sinc-general}
Assume that $f\in\MC_{\alpha}(\SEt(\domD_d))$ for $d$ with $0<d<\pi$.
Let $N$ be a positive integer,
and let $h$ be selected by the formula
\begin{equation}
h = \sqrt{\frac{\pi d}{\alpha N}}.
 \label{eq:h-SE}
\end{equation}
Then, there exists a constant $C$ independent of $N$ such that
\[
 \|f - \ProjSE f\|_{C([a,b])}
\leq C \sqrt{N}\E^{-\sqrt{\pi d \alpha N}}.
\]
\end{theorem}
\begin{theorem}[Okayama~{\cite[Theorem~6]{okayama13:_note}}]
\label{thm:DE-Sinc-general}
Assume that $f\in\MC_{\alpha}(\DEt(\domD_d))$ for $d$ with $0<d<\pi/2$.
Let $N$ be a positive integer,
and let $h$ be selected by the formula
\begin{equation}
 h = \frac{\log(2 d N/\alpha)}{N}.
 \label{eq:h-DE}
\end{equation}
Then, there exists a constant $C$ independent of $N$ such that
\[
 \|f - \ProjDE f\|_{C([a,b])}
\leq C \E^{-\pi d N/\log(2 d N/\alpha)}.
\]
\end{theorem}

\begin{remark}
In view of several studies~\cite{Shamloo,zarebnia13:_volter},
some researchers may labor under the false assumption that
$f$ is approximated as
\[
 f(t) \approx f(\SEt(-Nh))\omega_a(t)
+\sum_{j=-N+1}^{N-1} f(\SEt(jh)) S(j,h)(\SEtInv(t)) + f(\SEt(Nh))\omega_b(t).
\]
In reality, this is not true.
We should note that
``$f$ is approximated by a linear combination of
$\omega_a$, $S(-N+1,h),\,\ldots,\,S(N-1,h)$, $\omega_b$''
does not necessarily imply that
``the coefficients of the linear combination are
$f(\SEt(jh))$ ($j=-N,\,\ldots,\,N$).''
\end{remark}

\subsection{SE-Sinc quadrature and DE-Sinc quadrature}

The SE-Sinc and DE-Sinc quadratures are approximation formulas
for definite integrals,
and their convergence theorems can be stated as follows.
The dependency of the constants is explicitly expressed
following Okayama et al.~\cite{Okayama-et-al},
which is effectively used in a later analysis.

\begin{corollary}[Okayama et al.~{\cite[Corollary~2.1]{okayama11:_improv}}]
\label{cor:SE-Sinc-quadrature}
Assume that $f$ is analytic in $\SEt(\domD_d)$ for $d$ with $0<d<\pi$,
and there exists constants $L$ and $\alpha$ such that
\begin{equation}
 |f(z)|\leq L |z - a|^{\alpha - 1}|b - z|^{\alpha - 1}
\label{eq:LCminus1}
\end{equation}
holds for all $z\in\SEt(\domD_d)$.
Let $N$ be a positive integer,
and let $h$ be selected by the formula~\eqref{eq:h-SE}.
Then, there exists a constant $C_{\alpha,d}^{\textSE}$
depending only on $\alpha$ and $d$ such that
\[
\left|
 \int_a^b f(s)\D{s}
- h \sum_{j=-N}^N f(\SEt(jh))\SEtDiv(jh)
\right| \leq L (b - a)^{2\alpha-1} C_{\alpha,d}^{\textSE}
 \E^{-\sqrt{\pi d \alpha N}}.
\]
\end{corollary}

\begin{corollary}[Okayama et al.~{\cite[Corollary~2.3]{okayama11:_improv}}]
\label{cor:DE-Sinc-quadrature}
Assume that $f$ is analytic in $\DEt(\domD_d)$ for $d$ with $0<d<\pi/2$,
and there exist constants $L$ and $\alpha$ such that~\eqref{eq:LCminus1}
holds for all $z\in\DEt(\domD_d)$.
Let $N$ be a positive integer,
and let $h$ be selected by the formula~\eqref{eq:h-DE}.
Then, there exists a constant $C_{\alpha,d}^{\textDE}$
depending only on $\alpha$ and $d$ such that
\[
\left|
 \int_a^b f(s)\D{s}
- h \sum_{j=-N}^N f(\DEt(jh))\DEtDiv(jh)
\right| \leq L (b - a)^{2\alpha-1} C_{\alpha,d}^{\textDE}
 \E^{-2\pi d N/\log(2 d N/\alpha)}.
\]
\end{corollary}

We note that the selection formula of $h$ is not optimal here;
if $h$ is selected optimally, then the convergence rate becomes higher.
However, the optimal selection formula cannot be employed here because
the formula of $h$ should be unified in all approximations of
functions, definite integrals, and indefinite integrals.

\subsection{SE-Sinc indefinite integration and DE-Sinc indefinite integration}

The SE-Sinc and DE-Sinc indefinite integrations are approximation formulas
for indefinite integrals,
and their convergence theorems were provided as follows.
Here, $J(j,h)$ is defined by
\[
 J(j,h)(x) = h \left\{
\frac{1}{2} + \frac{1}{\pi}\Si\left[\frac{\pi(x - jh)}{h}\right]
\right\},
\]
where $\Si(x)$ is the sine integral defined by
$\Si(x)=\int_0^x\{(\sin t)/t\}\D{t}$.

\begin{theorem}[Okayama et al.~{\cite[Theorem~2.9]{Okayama-et-al}}]
\label{thm:SE-Sinc-indefinite}
Assume that $f$ is analytic in $\SEt(\domD_d)$ for $d$ with $0<d<\pi$,
and there exist constants $L$ and $\alpha$ such that~\eqref{eq:LCminus1}
holds for all $z\in\SEt(\domD_d)$.
Let $N$ be a positive integer,
and let $h$ be selected by the formula~\eqref{eq:h-SE}.
Then, there exists a constant $C$ independent of $N$ such that
\[
\max_{t\in[a, b]}\left|
 \int_a^t f(s)\D{s}
- \sum_{j=-N}^N f(\SEt(jh))\SEtDiv(jh)J(j,h)(\SEtInv(t))
\right| \leq C \E^{-\sqrt{\pi d \alpha N}}.
\]
\end{theorem}

\begin{theorem}[Okayama et al.~{\cite[Theorem~2.16]{Okayama-et-al}}]
\label{thm:DE-Sinc-indefinite}
Assume that $f$ is analytic in $\DEt(\domD_d)$ for $d$ with $0<d<\pi/2$,
and there exist constants $L$ and $\alpha$ such that~\eqref{eq:LCminus1}
holds for all $z\in\DEt(\domD_d)$.
Let $N$ be a positive integer,
and let $h$ be selected by the formula~\eqref{eq:h-DE}.
Then, there exists a constant $C$ independent of $N$ such that
\[
\max_{t\in[a,b]}
\left|
 \int_a^t f(s)\D{s}
- \sum_{j=-N}^N f(\DEt(jh))\DEtDiv(jh)J(j,h)(\DEtInv(t))
\right| \leq C \frac{\log(2 d N/\alpha)}{N} \E^{-\pi d N/\log(2 d N/\alpha)}.
\]
\end{theorem}

\section{Existing Sinc-collocation methods}
\label{sec:existing}

In this section, we describe two Sinc-collocation methods
proposed by Shamloo et al.~\cite{Shamloo}
and John--Ogbonna~\cite{JohnOgbo}.
The first method employs the SE transformation,
whereas the second method employs the DE transformation.

\subsection{Sinc-collocation method by Shamloo et al.}

First, we describe the Sinc-collocation method by Shamloo et al.
As outlined in Sect.~\ref{sec:intro},
after setting their approximate solution $\uSEn$ as~\eqref{eq:uSEn},
they substitute $\uSEn$ into the given equation~\eqref{eq:VFIE},
where the definite and indefinite integrals
are approximated by the SE-Sinc quadrature and
SE-Sinc indefinite integration, respectively.
For the sake of convenience,
we introduce the approximated integral operator
$\VolSEn: C([a, b])\to C([a, b])$
and $\FredSEn: C([a, b])\to C([a, b])$ as
\begin{align*}
\VolSEn[f](t)
&=\sum_{j=-N}^N k_1(t,\SEt(jh))f(\SEt(jh))\SEtDiv(jh)J(j,h)(\SEtInv(t)),\\
\FredSEn[f](t)
&=h\sum_{j=-N}^N k_2(t,\SEt(jh))f(\SEt(jh))\SEtDiv(jh).
\end{align*}
Using these notations, we obtain an equation expressed as
\[
 \uSEn(t) - \VolSEn[\uSEn](t) - \FredSEn[\uSEn](t) = g(t).
\]
Discretizing this equation at $n=2N+1$ collocation points $t = \SEt(ih)$
($i=-N,\,\ldots,\,N$), we obtain the system of equations
with respect to $\mathbd{c}_n=[c_{-N},\,\ldots,\,c_N]^{\mathrm{T}}$ as
\begin{equation}
 (\tilde{E}_n^{\textSE} - \tilde{V}_n^{\textSE} - \tilde{K}_n^{\textSE})\mathbd{c}_n
= \mathbd{g}_n^{\textSE},
\label{eq:linear-eq-Shamloo}
\end{equation}
where $\tilde{E}_n^{\textSE}$, $\tilde{V}_n^{\textSE}$
and $\tilde{K}_n^{\textSE}$ are $n\times n$ matrices defined by
\begin{align*}
 \tilde{E}_n^{\textSE}
&= \left[
   \begin{array}{@{\,}l|ccc|l@{\,}}
   \omega_a(\SEt(-Nh))        & 0    & \cdots &0  & \omega_b(\SEt(Nh)) \\
   \omega_a(\SEt(-(N-1)h)) & 1      &        &\OO & \omega_b(\SEt(-(N-1)h))\\
   \multicolumn{1}{c|}{\vdots}  &     & \ddots & & \multicolumn{1}{c}{\vdots} \\
   \omega_a(\SEt((N-1)h)) & \OO &        &1      & \omega_b(\SEt((N-1)h)) \\
   \omega_a(\SEt(Nh)) & 0      & \cdots &0      & \omega_b(\SEt(Nh))
   \end{array}
   \right], \\
  \tilde{V}_n^{\textSE}
&= \left[
   \begin{array}{@{\,}l|clc|l@{\,}}
    \VolSEn[\omega_a](\SEt(-Nh))
   &\cdots
   & k(\SEt(-Nh),\SEt(jh))\SEtDiv(jh) J(j,h)(-Nh)
   &\cdots
   & \VolSEn[\omega_b](\SEt(-Nh)) \\
    \multicolumn{1}{c|}{\vdots} & & \multicolumn{1}{c}{\vdots}
   & & \multicolumn{1}{c}{\vdots}\\
   \VolSEn[\omega_a](\SEt(Nh))
   &\cdots
   & k(\SEt(Nh),\SEt(jh))\SEtDiv(jh) J(j,h)(Nh)
   &\cdots
   & \VolSEn[\omega_b](\SEt(Nh))
   \end{array}
   \right],\\
  \tilde{K}_n^{\textSE}
&= \left[
   \begin{array}{@{\,}l|clc|l@{\,}}
    \FredSEn[\omega_a](\SEt(-Nh))
   &\cdots
   & k(\SEt(-Nh),\SEt(jh))\SEtDiv(jh) h
   &\cdots
   & \FredSEn[\omega_b](\SEt(-Nh)) \\
    \multicolumn{1}{c|}{\vdots} & & \multicolumn{1}{c}{\vdots}
   & & \multicolumn{1}{c}{\vdots}\\
   \FredSEn[\omega_a](\SEt(Nh))
   &\cdots
   & k(\SEt(Nh),\SEt(jh))\SEtDiv(jh) h
   &\cdots
   & \FredSEn[\omega_b](\SEt(Nh))
   \end{array}
   \right],
\end{align*}
and $\mathbd{g}_n^{\textSE}$ is an $n$-dimensional vector defined by
\[
 \mathbd{g}_n^{\textSE} =
[g(\SEt(-Nh)),\,\ldots,\,g(\SEt(Nh))]^{\mathrm{T}}.
\]
By solving the system~\eqref{eq:linear-eq-Shamloo},
the approximate solution $\uSEn$ is determined by~\eqref{eq:uSEn}.
To conduct its error analysis,
Shamloo et al.\ introduced the following function space.

\begin{definition}
Let $\alpha$ and $d$ be positive constants with
$\alpha\leq 1$ and $d\leq \pi/2$.
Let $\rho(z) = \E^{\SEtInv(z)} = (z - a)/(b - z)$.
Then, $\LC_{\alpha}(\SEt(\domD_d))$ denotes the family of all functions
$f$ analytic in $\SEt(\domD_d)$
for which there exists a constant $\tilde{L}$ such that
\begin{equation}
|f(z)| \leq \tilde{L} \frac{|\rho(z)|^{\alpha}}{(1 + |\rho(z)|)^{2\alpha}}
 \label{eq:LC}
\end{equation}
holds for all $z\in\domD$.
\end{definition}

Under this definition,
Shamloo et al.\ provided the following error analysis.

\begin{theorem}[Shamloo et al.~{\cite[Theorem~4.1]{Shamloo}}]
\label{thm:Shamloo}
Assume that the solution $u$ belongs to $\LC_{\alpha}(\SEt(\domD_d))$.
Furthermore, assume that
all the following functions:
\begin{align*}
 \frac{k_1(t, \cdot)\omega_a(\cdot)}{\left\{\SEtInv(\cdot)\right\}'},\quad
 \frac{k_1(t, \cdot)S(j,h)(\SEtInv(\cdot))}{\left\{\SEtInv(\cdot)\right\}'}
\,\, (j=-N+1,\,\ldots,\,N-1),\quad
 \frac{k_1(t, \cdot)\omega_b(\cdot)}{\left\{\SEtInv(\cdot)\right\}'},\\
 \frac{k_2(t, \cdot)\omega_a(\cdot)}{\left\{\SEtInv(\cdot)\right\}'},\quad
 \frac{k_2(t, \cdot)S(j,h)(\SEtInv(\cdot))}{\left\{\SEtInv(\cdot)\right\}'}
\,\, (j=-N+1,\,\ldots,\,N-1),\quad
 \frac{k_2(t, \cdot)\omega_b(\cdot)}{\left\{\SEtInv(\cdot)\right\}'},
\end{align*}
belong to $\LC_{\alpha}(\SEt(\domD_d))$ uniformly
for all $t\in [a, b]$ and positive integer $N$.
Let $h$ be selected by the formula~\eqref{eq:h-SE}.
Then, there exists a constant $C$ independent of $N$
such that~\eqref{eq:Shamloo-error-analysis} holds,
where $\mu_N=\|(\tilde{E}_n^{\textSE} - \tilde{V}_n^{\textSE} - \tilde{K}_n^{\textSE})^{-1}\|_2$.
\end{theorem}

There are two points to be discussed in this theorem.
First, $u\in\LC_{\alpha}(\SEt(\domD_d))$ is assumed,
but this assumption requires $u(a)=u(b)=0$
owing to the inequality~\eqref{eq:LC}.
This condition is not generally satisfied.
Furthermore, checking $u\in\LC_{\alpha}(\SEt(\domD_d))$ is,
in principle, difficult because $u$ is an unknown function to be determined.
If we do not know the appropriate values of
the parameters $\alpha$ and $d$, we cannot launch the method because
these values are indispensable in selecting the mesh size $h$
according to~\eqref{eq:h-SE}.
Second, the convergence of the method is not proved
because there exists a non-estimated term $\mu_N$,
which clearly depends on $N$.

Furthermore, the implementation of this method is rather complicated
because the first and last columns of the matrices
$\tilde{E}_n^{\textSE}$, $\tilde{V}_n^{\textSE}$ and $\tilde{K}_n^{\textSE}$
are inconsistent.
This is primarily because $\uSEn(\SEt(ih))=c_i$ does not hold.
In the next section, an improved Sinc-collocation method
with the SE transformation is derived,
which resolves the aforementioned issues on implementation and the theorem.

\subsection{Sinc-collocation method by John and Ogbonna}

Here, we describe the Sinc-collocation method by
John and Ogbonna~\cite{JohnOgbo}.
Note that as a variable transformation, they employed
a slightly different DE transformation
\[
 t =
\tilde{\psi}^{\textDE}(x)
 = \frac{b - a}{2}\tanh\left(\frac{\pi}{4}\sinh x\right)
+\frac{b+a}{2},
\]
and selected $h$ as
\[
 h = \frac{\log(4 d N/\alpha)}{N},
\]
which is also slightly different from~\eqref{eq:h-DE}.
For the sake of convenience,
we introduce the approximated integral operator
$\tilde{\mathcal{V}}_N^{\textDE}: C([a, b])\to C([a, b])$
and $\tilde{\mathcal{K}}_N^{\textDE}: C([a, b])\to C([a, b])$ as
\begin{align*}
\tilde{\mathcal{V}}_N^{\textDE}[f](t)
&=\sum_{j=-N}^N k_1(t,\tilde{\psi}^{\textDE}(jh))f(\tilde{\psi}^{\textDE}(jh))\{\tilde{\psi}^{\textDE}\}'(jh)J(j,h)(\{\tilde{\psi}^{\textDE}\}^{-1}(t)),\\
\tilde{\mathcal{K}}_N^{\textDE}[f](t)
&=h\sum_{j=-N}^N k_1(t,\tilde{\psi}^{\textDE}(jh))f(\tilde{\psi}^{\textDE}(jh))\{\tilde{\psi}^{\textDE}\}'(jh).
\end{align*}
They set their approximate solution $\uDEn$ as
\begin{equation}
 \uDEn(t) = c_{-N}\omega_a(t)
+\sum_{j=-N+1}^{N-1} c_j S(j,h)(\{\tilde{\psi}^{\textDE}\}^{-1}(t)) + c_{N} \omega_b(t).
\label{eq:uDEn}
\end{equation}
Then, they substitute $\uDEn$ into the given equation~\eqref{eq:VFIE},
where the definite and indefinite integrals
are approximated by the DE-Sinc quadrature and
DE-Sinc indefinite integration, respectively.
The obtained equation is expressed as
\[
 \uDEn(t) - \tilde{\mathcal{V}}_N^{\textDE}[\uDEn](t) - \tilde{\mathcal{K}}_N^{\textDE}[\uDEn](t) = g(t).
\]
Discretizing this equation at $n=2N+1$ collocation points
$t=t_i^{\textDE}$ $(i=-N,\,\ldots,\,N)$, where
\begin{equation}
\label{eq:inconsistent-points}
t_i^{\textDE} =
 \begin{cases}
     a & (i = -N),\\
\tilde{\psi}^{\textDE}(ih) & (i= -N+1,\,\ldots,\,N-1),\\
     b & (i = N),
    \end{cases}
\end{equation}
we obtain the system of equations
with respect to $\mathbd{c}_n=[c_{-N},\,\ldots,\,c_N]^{\mathrm{T}}$ as
\begin{equation}
 (\tilde{E}_n^{\textDE} - \tilde{V}_n^{\textDE} - \tilde{K}_n^{\textDE})\mathbd{c}_n
= \tilde{\mathbd{g}}_n^{\textDE},
\label{eq:linear-eq-JohnOgbo}
\end{equation}
where $\tilde{E}_n^{\textDE}$, $\tilde{V}_n^{\textDE}$
and $\tilde{K}_n^{\textDE}$ are $n\times n$ matrices defined by
\begin{align*}
 \tilde{E}_n^{\textDE}
&= \left[
   \begin{array}{@{\,}l|ccc|l@{\,}}
   \multicolumn{1}{c|}{1}        & 0    & \cdots &0  & \multicolumn{1}{c}{0} \\
   \hline
   \omega_a(t^{\textDE}_{-N+1}) & 1      &        &\OO & \omega_b(t^{\textDE}_{-N+1})\\
   \multicolumn{1}{c|}{\vdots}  &     & \ddots & & \multicolumn{1}{c}{\vdots} \\
   \omega_a(t^{\textDE}_{N-1}) & \OO &        &1      & \omega_b(t^{\textDE}_{N-1}) \\
   \hline
   \multicolumn{1}{c|}{0} & 0      & \cdots &0      & \multicolumn{1}{c}{1}
   \end{array}
   \right], \\
  \tilde{V}_n^{\textDE}
&= \left[
   \begin{array}{@{\,}l|clc|l@{\,}}
\multicolumn{1}{c|}{0} & \cdots & \multicolumn{1}{c}{0} & \cdots &
\multicolumn{1}{c}{0} \\
\hline
    \tilde{\mathcal{V}}_N^{\textDE}[\omega_a](t_{-N+1}^{\textDE})
   &\cdots
   & k(t^{\textDE}_{-N+1},t^{\textDE}_j)\{\tilde{\psi}^{\textDE}\}'(jh) J(j,h)(-(N-1)h)
   &\cdots
   & \tilde{\mathcal{V}}_N^{\textDE}[\omega_b](t^{\textDE}_{-N+1}) \\
    \multicolumn{1}{c|}{\vdots} & & \multicolumn{1}{c}{\vdots}
   & & \multicolumn{1}{c}{\vdots}\\
   \tilde{\mathcal{V}}_N^{\textDE}[\omega_a](t^{\textDE}_{N-1})
   &\cdots
   & k(t^{\textDE}_{N-1},t^{\textDE}_j)\{\tilde{\psi}^{\textDE}\}'(jh) J(j,h)((N-1)h)
   &\cdots
   & \tilde{\mathcal{V}}_N^{\textDE}[\omega_b](t^{\textDE}_{N-1}) \\
 \hline
\tilde{\mathcal{V}}_N^{\textDE}[\omega_a](b) & \cdots & k(b,t^{\textDE}_j)\{\tilde{\psi}^{\textDE}\}'(jh) h & \cdots & \tilde{\mathcal{V}}_N^{\textDE}[\omega_b](b)
   \end{array}
   \right],\\
  \tilde{K}_n^{\textDE}
&= \left[
   \begin{array}{@{\,}l|clc|l@{\,}}
    \tilde{\mathcal{K}}_N^{\textDE}[\omega_a](a)
   &\cdots
   & k(a,t^{\textDE}_j)\{\tilde{\psi}^{\textDE}\}'(jh) h
   &\cdots
   & \tilde{\mathcal{K}}_N^{\textDE}[\omega_b](a) \\
 \hline
    \tilde{\mathcal{K}}_N^{\textDE}[\omega_a](t^{\textDE}_{-N+1})
   &\cdots
   & k(t^{\textDE}_{-N+1},t^{\textDE}_j)\{\tilde{\psi}^{\textDE}\}'(jh) h
   &\cdots
   & \tilde{\mathcal{K}}_N^{\textDE}[\omega_b](t^{\textDE}_{-N+1}) \\
    \multicolumn{1}{c|}{\vdots} & & \multicolumn{1}{c}{\vdots}
   & & \multicolumn{1}{c}{\vdots}\\
   \tilde{\mathcal{K}}_N^{\textDE}[\omega_a](t^{\textDE}_{N-1})
   &\cdots
   & k(t^{\textDE}_{N-1},t^{\textDE}_j)\{\tilde{\psi}^{\textDE}\}'(jh) h
   &\cdots
   & \tilde{\mathcal{K}}_N^{\textDE}[\omega_b](t^{\textDE}_{N-1})\\
\hline
   \tilde{\mathcal{K}}_N^{\textDE}[\omega_a](b)
   &\cdots
   & k(b,t^{\textDE}_j)\{\tilde{\psi}^{\textDE}\}'(jh) h
   &\cdots
   & \tilde{\mathcal{K}}_N^{\textDE}[\omega_b](b)
   \end{array}
   \right],
\end{align*}
and $\tilde{\mathbd{g}}_n^{\textDE}$ is an $n$-dimensional vector defined by
\[
 \tilde{\mathbd{g}}_n^{\textDE} =
[g(a),\,g(t^{\textDE}_{-N+1}),\,\ldots,\,g(t^{\textDE}_{N-1}),\,g(b)]^{\mathrm{T}}.
\]
By solving the system~\eqref{eq:linear-eq-JohnOgbo},
the approximate solution $\uDEn$ is determined by~\eqref{eq:uDEn}.
Its error analysis has not been explicitly provided.

Owing to the inconsistency of
their collocation points~\eqref{eq:inconsistent-points},
the first and last rows of the matrices of
$\tilde{E}_n^{\textDE}$, $\tilde{V}_n^{\textDE}$
and $\tilde{K}_n^{\textDE}$
are inconsistent, in addition to the first and last columns.
In the next section, this Sinc-collocation method
with the DE transformation is also improved to facilitate its implementation,
and its convergence theorem is provided.

\section{New Sinc-collocation methods}
\label{sec:new}

In this section, we describe
improvements of the two existing Sinc-collocation methods
described in the previous section.
The associated convergence theorems are also presented.

\subsection{New Sinc-collocation method with the SE transformation}
\label{sec:new-SE}

Here, we improve the Sinc-collocation method by Shamloo et al.
The main idea behind this improvement is in the approximate solution;
we set $\vSEn$ as~\eqref{eq:vSEn}.
Then, we substitute $\vSEn$ into the given equation~\eqref{eq:VFIE},
where the definite and indefinite integrals
are approximated by the SE-Sinc quadrature and
SE-Sinc indefinite integration, respectively.
The obtained equation is expressed as
\[
 \vSEn(t) - \VolSEn[\vSEn](t) - \FredSEn[\vSEn](t) = g(t).
\]
Discretizing this equation at $n=2N+1$ collocation points $t = \SEt(ih)$
($i=-N,\,\ldots,\,N$), we obtain the system of equations
with respect to $\mathbd{c}_n=[c_{-N},\,\ldots,\,c_N]^{\mathrm{T}}$ as
\begin{equation}
 (I_n - V_n^{\textSE} - K_n^{\textSE})\mathbd{c}_n
= \mathbd{g}_n^{\textSE},
\label{eq:linear-eq-SE-Sinc}
\end{equation}
where $I_n$ is an identity matrix of order $n$,
and $V_n^{\textSE}$
and $K_n^{\textSE}$ are $n\times n$ matrices whose $(i, j)$-th element is
\begin{align*}
 \left(V_n^{\textSE}\right)_{ij}
&= k_1(\SEt(ih),\SEt(jh))\SEtDiv(jh)J(j,h)(ih),
\quad i=-N,\,\ldots,\,N,\quad j=-N,\,\ldots,\,N,\\
 \left(K_n^{\textSE}\right)_{ij}
&= k_2(\SEt(ih),\SEt(jh))\SEtDiv(jh)h,
\quad i=-N,\,\ldots,\,N,\quad j=-N,\,\ldots,\,N.
\end{align*}
By solving the system~\eqref{eq:linear-eq-SE-Sinc},
the approximate solution $\vSEn$ is determined by~\eqref{eq:vSEn}.
For this method,
we present the following convergence theorem.
The proof is provided in Sect.~\ref{sec:proofs}.
Here, $\Vol$ and $\Fred$ are integral operators defined by
\begin{align*}
 \Vol[f](t) &= \int_a^t k_1(t,s) f(s)\D{s},\\
 \Fred[f](t) &= \int_a^b k_2(t,s) f(s)\D{s}.
\end{align*}

\begin{theorem}
\label{thm:SE-Sinc-collocation}
Let $\alpha$ and $d$ be positive constants with $\alpha\leq 1$
and $d<\pi$.
Assume that $k_1(z,\cdot)$ and $k_2(z,\cdot)$
belong to $\HC(\SEt(\domD_d))$
uniformly for $z\in\overline{\SEt(\domD_d)}$,
and $g$, $k_1(\cdot,w)$ and $k_2(\cdot,w)$
belong to $\MC_{\alpha}(\SEt(\domD_d))$
uniformly for $w\in\overline{\SEt(\domD_d)}$.
Furthermore, assume that the homogeneous equation
$(I - \Vol - \Fred)f = 0$ has only the trivial solution $f\equiv 0$.
Let $h$ be selected by the formula~\eqref{eq:h-SE}.
Then, there exists a positive integer $N_0$ such that
for all $N\geq N_0$,
the coefficient matrix $(I_n - V_n^{\textSE} - K_n^{\textSE})$
is invertible.
Furthermore, there exists a constant $C$ independent of $N$ such that
for all $N\geq N_0$,
\[
 \|u - \vSEn\|_{C([a,b])}
\leq C\sqrt{N}\E^{-\sqrt{\pi d \alpha N}}.
\]
\end{theorem}

In this theorem, no condition is assumed on the solution $u$,
whereas $u\in\LC_{\alpha}(\SEt(\domD_d))$ is assumed
in Theorem~\ref{thm:Shamloo}.
Furthermore, this theorem rigorously proves root-exponential convergence,
whereas Theorem~\ref{thm:Shamloo} does not because of $\mu_N$.

\subsection{New Sinc-collocation method with the DE transformation}

Here, we improve the Sinc-collocation method by
John and Ogbonna~\cite{JohnOgbo}.
The main idea behind this improvement is to replace the SE transformation
in the new Sinc-collocation method (described in Sect.~\ref{sec:new-SE})
with the DE transformation.
We introduce the approximated integral operator
$\VolDEn: C([a, b])\to C([a, b])$
and $\FredDEn: C([a, b])\to C([a, b])$ as
\begin{align*}
\VolDEn[f](t)
&=\sum_{j=-N}^N k_1(t,\DEt(jh))f(\DEt(jh))\DEtDiv(jh)J(j,h)(\DEtInv(t)),\\
\FredDEn[f](t)
&=h\sum_{j=-N}^N k_1(t,\DEt(jh))f(\DEt(jh))\DEtDiv(jh).
\end{align*}
After setting the approximate solution $\vDEn$ as
\begin{equation}
 \vDEn(t)
= c_{-N}\omega_a(t)
 + \sum_{j=-N}^N\left\{c_j - c_{-N}\omega_a(\DEt(jh)) - c_N\omega_b(\DEt(jh))\right\}S(j,h)(\DEtInv(t))
+ c_N\omega_b(t),
\label{eq:vDEn}
\end{equation}
we substitute $\vDEn$ into the given equation~\eqref{eq:VFIE},
where the definite and indefinite integrals
are approximated by the DE-Sinc quadrature and
DE-Sinc indefinite integration, respectively.
The obtained equation is expressed as
\[
 \vDEn(t) - \VolDEn[\vDEn](t) - \FredDEn[\vDEn](t) = g(t).
\]
Discretizing this equation at $n=2N+1$ collocation points $t = \DEt(ih)$
($i=-N,\,\ldots,\,N$), we obtain the system of equations
with respect to $\mathbd{c}_n=[c_{-N},\,\ldots,\,c_N]^{\mathrm{T}}$ as
\begin{equation}
 (I_n - V_n^{\textDE} - K_n^{\textDE})\mathbd{c}_n
= \mathbd{g}_n^{\textDE},
\label{eq:linear-eq-DE-Sinc}
\end{equation}
where $V_n^{\textDE}$
and $K_n^{\textDE}$ are $n\times n$ matrices whose $(i, j)$-th element is
\begin{align*}
 \left(V_n^{\textDE}\right)_{ij}
&= k_1(\DEt(ih),\DEt(jh))\DEtDiv(jh)J(j,h)(ih),
\quad i=-N,\,\ldots,\,N,\quad j=-N,\,\ldots,\,N,\\
 \left(K_n^{\textDE}\right)_{ij}
&= k_2(\DEt(ih),\DEt(jh))\DEtDiv(jh)h,
\quad i=-N,\,\ldots,\,N,\quad j=-N,\,\ldots,\,N,
\end{align*}
and $\mathbd{g}_n^{\textDE}$ is an $n$-dimensional vector defined by
\[
 \mathbd{g}_n^{\textDE} =
[g(\DEt(-Nh)),\,\ldots,\,g(\DEt(Nh))]^{\mathrm{T}}.
\]
By solving the system~\eqref{eq:linear-eq-DE-Sinc},
the approximate solution $\vDEn$ is determined by~\eqref{eq:vDEn}.
For this method,
we present the following convergence theorem.
The proof is provided in Sect.~\ref{sec:proofs}.

\begin{theorem}
\label{thm:DE-Sinc-collocation}
Let $\alpha$ and $d$ be positive constants with $\alpha\leq 1$
and $d<\pi/2$.
Assume that $k_1(z,\cdot)$ and $k_2(z,\cdot)$
belong to $\HC(\DEt(\domD_d))$
uniformly for $z\in\overline{\DEt(\domD_d)}$,
and $g$, $k_1(\cdot,w)$ and $k_2(\cdot,w)$
belong to $\MC_{\alpha}(\DEt(\domD_d))$
uniformly for $w\in\overline{\DEt(\domD_d)}$.
Furthermore, assume that the homogeneous equation
$(I - \Vol - \Fred)f = 0$ has only the trivial solution $f\equiv 0$.
Let $h$ be selected by the formula~\eqref{eq:h-DE}.
Then, there exists a positive integer $N_0$ such that
for all $N\geq N_0$,
the coefficient matrix $(I_n - V_n^{\textDE} - K_n^{\textDE})$
is invertible.
Furthermore, there exists a constant $C$ independent of $N$ such that
for all $N\geq N_0$,
\[
 \|u - \vDEn\|_{C([a,b])}
\leq C\E^{-\pi d N/\log(2 d N/\alpha)}.
\]
\end{theorem}

The convergence rate of this theorem
is significantly higher than that of Theorem~\ref{thm:SE-Sinc-collocation}.

\section{Numerical examples}
\label{sec:numer}

This section presents numerical results.
All the programs were written in C with double-precision floating-point
arithmetic.
The source code is available at
\url{https://github.com/okayamat/sinc-colloc-vol-fred}.

First, we consider the following equation,
which is relatively easy to solve numerically.

\begin{example}[Shamloo et al.~{\cite[Example~5.1]{Shamloo}}]
\label{ex:1}
Consider the following equation
\[
 u(t) - \int_0^t (ts) u(s)\D{s} - \int_0^1 (ts) u(s)\D{s}
= \frac{2}{3}t - \frac{1}{3}t^4,\quad 0\leq t\leq 1,
\]
whose solution is $u(t)=t$.
\end{example}

In this case, the assumptions of Theorem~\ref{thm:SE-Sinc-collocation}
are fulfilled with $\alpha=1$ and $d=3.14$ (slightly less than $\pi$),
and those of Theorem~\ref{thm:DE-Sinc-collocation}
are fulfilled with $\alpha=1$ and $d=1.57$ (slightly less than $\pi/2$).
Therefore, these values were used for the selection formulas
of $h$ in~\eqref{eq:h-SE} and~\eqref{eq:h-DE}, respectively.
For the original Sinc-collocation methods,
Shamloo et al.~\cite{Shamloo} fixed
the formula of $h$ as $h = \pi/\sqrt{N}$,
and John--Ogbonna~\cite{JohnOgbo} fixed
the formula of $h$ as $h = \log(\pi N)/N$.
Therefore, these formulas were used for computation
in the original Sinc-collocation methods.
The error was examined at 4096 equally spaced points
over the given interval, and the maximum error was plotted
on the graph.
The result is shown in Fig.~\ref{fig:example1}.
In the top graph, we observe that
the Sinc-collocation methods with the DE transformation
converge much faster than those with the SE transformation.
The bottom graph shows the condition number (with the infinity norm)
of the resulting linear system.
The result indicates that the system is well-conditioned
for all methods.

Next, we consider the following equation,
which is not easy to solve numerically, particularly
for Gaussian quadrature-based methods.
This is due to the derivative singularity at the endpoints.

\begin{example}
\label{ex:2}
Consider the following equation
\[
 u(t) - \int_0^t s^{t+1/2} u(s)\D{s} - \int_0^1 (1-s)^t u(s)\D{s}
= \sqrt{t} - \frac{t^{t+2}}{t+2} - \Beta\left(\frac{3}{2},t+1\right),
\quad 0\leq t\leq 1,
\]
whose solution is $u(t)=\sqrt{t}$.
Here, $\Beta(\alpha,\beta)$ is the beta function.
\end{example}

In this case, the assumptions of Theorem~\ref{thm:SE-Sinc-collocation}
are fulfilled with $\alpha=1/2$ and $d=3.14$,
and those of Theorem~\ref{thm:DE-Sinc-collocation}
are fulfilled with $\alpha=1/2$ and $d=1.11$.
Therefore, these values were used for the selection formulas
of $h$ in~\eqref{eq:h-SE} and~\eqref{eq:h-DE}, respectively.
The error was examined at 4096 equally spaced points
over the given interval, and the maximum error was plotted
on the graph.
The result is shown in Fig.~\ref{fig:example2},
which exhibits a pattern similar to Example~\ref{ex:1}.
It should be emphasized that
the convergence of the original Sinc-collocation methods is not proved,
whereas that of the new Sinc-collocation methods is proved
by Theorems~\ref{thm:SE-Sinc-collocation}
and~\ref{thm:DE-Sinc-collocation}.
Furthermore,
implementation of the original Sinc-collocation methods
is rather complicated because there exists an
inconsistency in the coefficient matrices.

Finally, we consider the following equation
from a recent paper~\cite[Example 1]{fermo25:_volter_fredh},
where the kernel $k_1(t,s)$ is slightly modified
so that it has derivative singularity at $s=-1$.

\begin{example}
\label{ex:3}
Consider the following equation
\begin{align*}
& u(t) - \frac{1}{2\pi}\int_{-1}^t t \E^{-s} u(s) (t - s)\sqrt{1 + s}\D{s}
 + \frac{1}{\pi}\int_{-1}^1 (s+t^2) u(s) \sqrt{1 - s^2}\D{s}\\
&= \sqrt{1+t} - \frac{t}{2\pi}\left\{(t+3)\E^{-t} + \E (t - 1)\right\}
+\frac{16\sqrt{2}}{105\pi}(7t^2 + 1),
\quad -1\leq t\leq 1,
\end{align*}
whose solution is $u(t)=\sqrt{1+t}$.
\end{example}

In this case, the assumptions of Theorem~\ref{thm:SE-Sinc-collocation}
are fulfilled with $\alpha=1/2$ and $d=3.14$,
and those of Theorem~\ref{thm:DE-Sinc-collocation}
are fulfilled with $\alpha=1/2$ and $d=1.57$.
Therefore, these values were used for the selection formulas
of $h$ in~\eqref{eq:h-SE} and~\eqref{eq:h-DE}, respectively.
The error was examined at 4096 equally spaced points
over the given interval, and the maximum error was plotted
on the graph.
The result is shown in Fig.~\ref{fig:example3},
in a double logarithmic chart.
In this figure, the result of an existing method based on
the product and Gauss rules~\cite{fermo25:_volter_fredh}
(implemented in MATLAB) is also shown.
In this graph as well, we observe that
the Sinc-collocation methods with the DE transformation
converge much faster than those with the SE transformation.
Furthermore, we observe that
the Sinc-collocation methods converge exponentially (the line is curved),
whereas the existing method converges polynomially (the line is straight).
The condition number of all the systems remains at a low level.

\begin{figure}[htpb]
{\centering
\includegraphics[scale=.75]{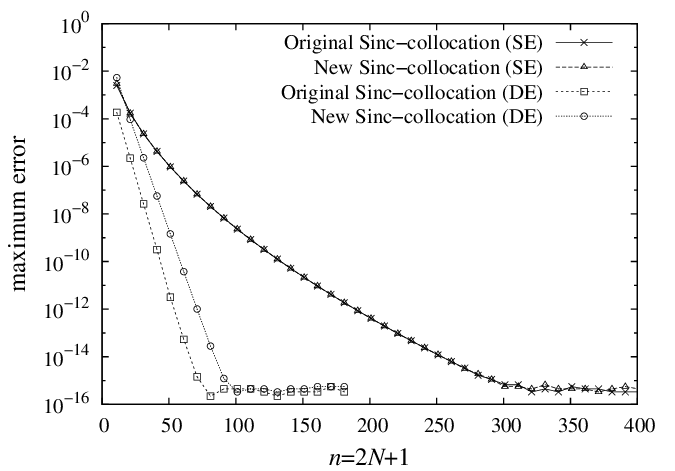}\\
\includegraphics[scale=.75]{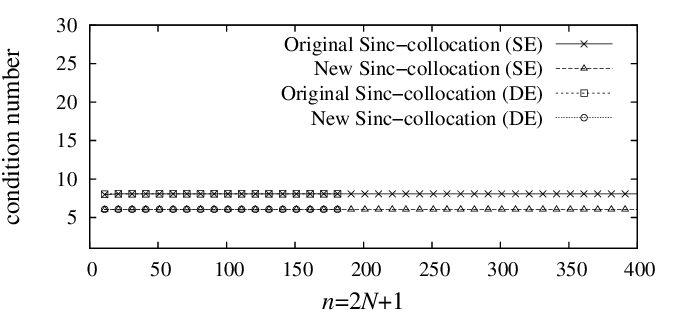}
\caption{Results of Example~\ref{ex:1}: (top) convergence profile; (bottom) condition number of the resulting matrix.}\label{fig:example1}
}
\end{figure}
\begin{figure}[htpb]
{\centering
\includegraphics[scale=.75]{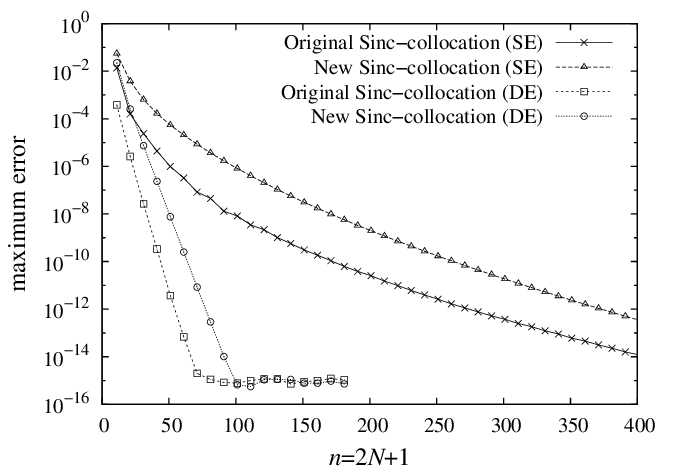}\\
\includegraphics[scale=.75]{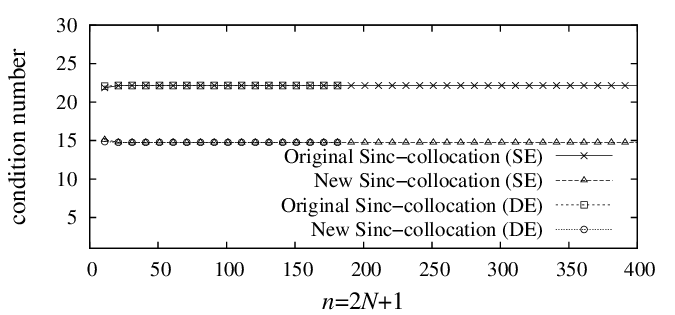}
\caption{Results of Example~\ref{ex:2}: (top) convergence profile; (bottom) condition number of the resulting matrix.}\label{fig:example2}
}
\end{figure}
\begin{figure}[htpb]
{\centering
\includegraphics[scale=.75]{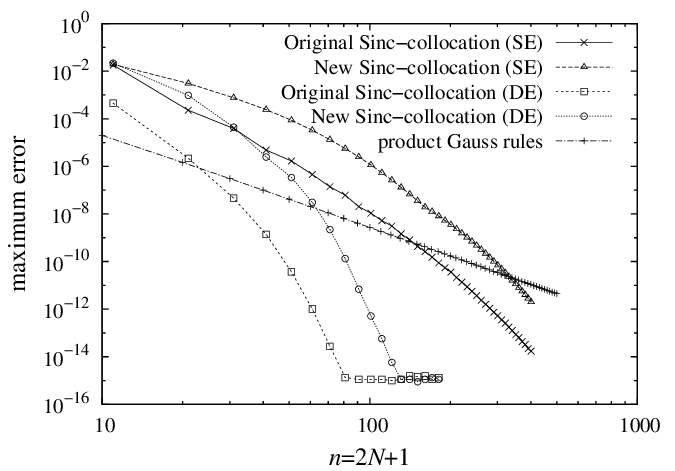}\\
\includegraphics[scale=.75]{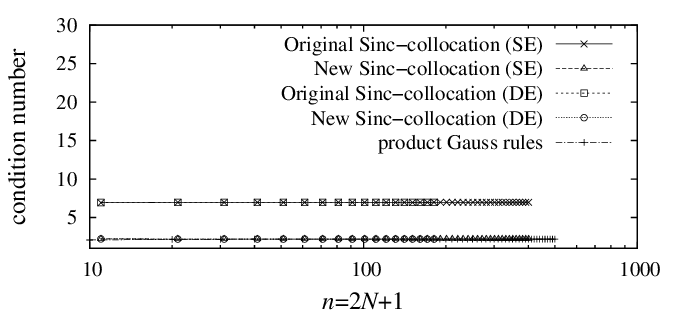}
\caption{Results of Example~\ref{ex:3}: (top) convergence profile; (bottom) condition number of the resulting matrix.}\label{fig:example3}
}
\end{figure}

\section{Proofs}
\label{sec:proofs}

In this section,
proofs of the new theorems stated in Sect.~\ref{sec:new} are provided.

\subsection{Sketch of the proof of Theorem~\ref{thm:SE-Sinc-collocation}}
\label{sec:sketch-SE}

Here, for readers' convenience, we present
the organization of the proof of Theorem~\ref{thm:SE-Sinc-collocation}.
In this proof, we consider the following three equations:
\begin{align}
 (I - \Vol - \Fred) u &= g, \nonumber\\
 (I - \VolSEn - \FredSEn) w &= g, \label{eq:SE-Nystroem-symbol}\\
 (I - \ProjSE\VolSEn - \ProjSE\FredSEn) v &= \ProjSE g.
 \label{eq:SE-collocation-symbol}
\end{align}
In Sect.~\ref{sec:regularity},
we analyze the first equation,
which is nothing but~\eqref{eq:VFIE},
and show that the equation has
a unique solution $u\in\MC_{\alpha}(\SEt(\domD_d))$.
In Sect.~\ref{sec:SE-Nystroem},
we analyze the second equation,
and show that~\eqref{eq:SE-Nystroem-symbol} has
a unique solution $w\in C([a,b])$, and there exists a
constant $C$ independent of $N$ such that
\begin{equation}
 \|u - w\|_{C([a,b])}\leq C \E^{-\sqrt{\pi d \alpha N}}.
\label{eq:error-SE-Nystroem}
\end{equation}
In Sect.~\ref{sec:SE-collocation},
we analyze the third equation,
and show that~\eqref{eq:SE-collocation-symbol} has
a unique solution $v \in C([a,b])$,
which means that the system~\eqref{eq:linear-eq-SE-Sinc} is
uniquely solvable, i.e.,
the coefficient matrix $(I_n - V_n^{\textSE} - K_n^{\textSE})$
is invertible.
Furthermore, we show that $v = \vSEn = \ProjSE w$ holds.
Thus, we have
\begin{align}
 \|u - \vSEn\|_{C([a,b])}
\leq \|u - \ProjSE u\|_{C([a,b])}
 + \|\ProjSE u - \ProjSE w\|_{C([a,b])}
\leq \|u - \ProjSE u\|_{C([a,b])}
 + \|\ProjSE\|_{\mathcal{L}(C([a,b]),C([a,b]))} \| u - w\|_{C([a,b])}.
\label{eq:final-error-SE}
\end{align}
Because $u\in\MC_{\alpha}(\SEt(\domD_d))$,
we can use Theorem~\ref{thm:SE-Sinc-general} for the first term.
For the second term, we can use~\eqref{eq:error-SE-Nystroem}
and the following lemma.

\begin{lemma}[Okayama~{\cite[Lemma~7.2]{OkayamaFredholm}}]
Let $\ProjSE$ be defined by~\eqref{eq:ProjSE}.
Then, there exists a constant $C$ independent of $N$ such that
\[
 \|\ProjSE\|_{\mathcal{L}(C([a,b]),C([a,b]))}
\leq C\log(N+1).
\]
\end{lemma}

Thus, we obtain the desired error bound
of Theorem~\ref{thm:SE-Sinc-collocation}.

\subsection{On the first equation~\eqref{eq:VFIE}}
\label{sec:regularity}

Here, we prove $u\in\MC_{\alpha}(\domD)$
($\domD$ is either $\SEt(\domD_d)$ or $\DEt(\domD_d)$).
The important tool here is the Fredholm alternative theorem.
We start with showing the following lemma.

\begin{lemma}
\label{lem:compact}
Assume that $k_1(z,\cdot)$ and $k_2(z,\cdot)$
belong to $\HC(\domD)$ for all $z\in\overline{\domD}$,
and $k_1(\cdot,w)$ and $k_2(\cdot,w)$
belong to $\HC(\domD)$ for all $w\in\overline{\domD}$.
Then, the operator $(\Vol + \Fred) :\HC(\domD)\to \HC(\domD)$ is compact.
\end{lemma}
\begin{proof}
It is easy to show that
the operator $(\Vol + \Fred)$ maps the set $\{f : \|f\|_{\HC(\domD)}\leq 1\}$
onto a uniformly bounded and equicontinuous set.
Thus, the claim follows from the Arzel\`{a}--Ascoli theorem
for complex functions (cf. Rudin~\cite[Theorem 11.28]{rudin87:_real}).
\end{proof}

Therefore, we obtain the following result
by the Fredholm alternative theorem.

\begin{theorem}
\label{thm:u-HC}
Assume that all the assumptions of Lemma~\ref{lem:compact} are fulfilled.
Furthermore, assume that
the homogeneous equation $(I - \Vol - \Fred)f = 0$
has only the trivial solution $f\equiv 0$.
Then, the operator $(I - \Vol - \Fred):\HC(\domD)\to\HC(\domD)$
has a bounded inverse,
$(I - \Vol - \Fred)^{-1}:\HC(\domD)\to\HC(\domD)$.
Furthermore, if $g\in\HC(\domD)$,
then~\eqref{eq:VFIE} has a unique solution $u\in\HC(\domD)$.
\end{theorem}

For $u\in\MC_{\alpha}(\domD)$,
we must further show the H\"{o}lder continuity at the endpoints.
This is done in the proof of the next theorem.

\begin{theorem}
Assume that all the assumptions of Theorem~\ref{thm:u-HC} are fulfilled.
Furthermore, assume that
$g$, $k_1(\cdot,w)$, and $k_2(\cdot,w)$
belong to $\MC_{\alpha}(\SEt(\domD_d))$
uniformly for $w\in\overline{\SEt(\domD_d)}$.
Then,~\eqref{eq:VFIE} has a unique solution $u\in\MC_{\alpha}(\domD)$.
\end{theorem}
\begin{proof}
From Theorem~\ref{thm:u-HC}, we have $u\in\HC(\domD)$.
Therefore, it suffices to show that there exists a constant $C$ such that
$|u(z) - u(a)|\leq C |z - a|^{\alpha}$
and $|u(b) - u(z)|\leq C |b - z|^{\alpha}$ hold
for all $z\in\overline{\domD}$.
From $u = g + \Vol u + \Fred u$, we have
\begin{align*}
& |u(b) - u(z)|\\
&\leq
|g(b) - g(z)| + \int_z^{b}|k_1(b,w)| |u(w)| |\mathrm{d}{w}|
+\int_a^z|k_1(b,w) - k_1(z,w)| |u(w)||\mathrm{d}{w}|
+\int_a^b |k_2(b,w) - k_2(z,w)| |u(w)||\mathrm{d}{w}|.
\end{align*}
Thus, noting $\alpha\leq 1$, we obtain $|u(b) - u(z)|\leq C_1|b-z|^{\alpha}$
by the H\"{o}lder continuity of $g$, $k_1(\cdot,w)$ and $k_2(\cdot,w)$,
and by the boundedness of $\domD$, $k_1(b,\cdot)$ and $u$.
In the same manner,
we obtain $|u(z) - u(a)|\leq C_2|z - a|^{\alpha}$.
Setting $C = \max\{C_1,\,C_2\}$, we obtain the desired inequalities.
\end{proof}

In the same manner as Theorem~\ref{thm:u-HC},
$u\in C([a,b])$ is also established as follows.
The proof is omitted.

\begin{theorem}
\label{thm:u-C}
Assume that $k_1(t,\cdot)$ and $k_2(t,\cdot)$
belong to $C([a,b])$ for all $t\in [a,b]$,
and $k_1(\cdot,s)$ and $k_2(\cdot,s)$
belong to $C([a,b])$ for all $s\in [a,b]$.
Furthermore, assume that
the homogeneous equation $(I - \Vol - \Fred)f = 0$
has only the trivial solution $f\equiv 0$.
Then, the operator $(I - \Vol - \Fred): C([a,b])\to C([a,b])$
has a bounded inverse,
$(I - \Vol - \Fred)^{-1}: C([a,b])\to C([a,b])$.
Furthermore, if $g\in C([a,b])$,
then~\eqref{eq:VFIE} has a unique solution $u\in C([a,b])$.
\end{theorem}

\subsection{On the second equation~\eqref{eq:SE-Nystroem-symbol}}
\label{sec:SE-Nystroem}

Next, we show that~\eqref{eq:SE-Nystroem-symbol} has
a unique solution $w\in C([a,b])$, and estimate the term
$\|u - w\|_{C([a,b])}$.
For this purpose,
the following perturbation theorem is useful.

\begin{theorem}[Atkinson~{\cite[Theorem~4.1.1]{atkinson97:_numer_solut}}]
\label{thm:Atkinson}
Assume the following conditions:
\begin{enumerate}
 \item Operators $\mathcal{X}$ and $\mathcal{X}_n$
are bounded operators on $C([a,b])$ to $C([a,b])$.
 \item Operator $(I - \mathcal{X}):C([a,b])\to C([a,b])$
has a bounded inverse
$(I - \mathcal{X})^{-1}:C([a,b])\to C([a,b])$.
 \item Operator $\mathcal{X}_n$ is compact on $C([a,b])$.
 \item The following inequality holds:
\begin{equation}
\label{eq:operator-converge-condition}
\|(\mathcal{X}-\mathcal{X}_n)\mathcal{X}_n\|_{\mathcal{L}(C([a,b]),C([a,b]))}
<\frac{1}{\|(I - \mathcal{X})^{-1}\|_{\mathcal{L}(C([a,b]),C([a,b]))}}.
\end{equation}
\end{enumerate}
Then, $(I - \mathcal{X}_n)^{-1}$ exists as a bounded operator
on $C([a,b])$ to $C([a,b])$, with
\begin{equation}
\|(I - \mathcal{X}_n)^{-1}\|_{\mathcal{L}(C([a,b]),C([a,b]))}
\leq \frac{1 + \|(I - \mathcal{X})^{-1}\|_{\mathcal{L}(C([a,b]),C([a,b]))}
\|\mathcal{X}_n\|_{\mathcal{L}(C([a,b]),C([a,b]))}}
{1 - \|(I - \mathcal{X})^{-1}\|_{\mathcal{L}(C([a,b]),C([a,b]))}\|(\mathcal{X}-\mathcal{X}_n)\mathcal{X}_n\|_{\mathcal{L}(C([a,b]),C([a,b]))}}.
\label{InEq:Bound-Inverse-Op}
\end{equation}
Furthermore, if $(I - \mathcal{X})u = g$ and
$(I - \mathcal{X}_n)w = g$, then
\[
 \|u - w\|_{C([a,b])}\leq \|(I - \mathcal{X}_n)^{-1}\|_{\mathcal{L}(C([a,b]),C([a,b]))}
\|(\mathcal{X} - \mathcal{X}_n)u\|_{C([a,b])}.
\]
\end{theorem}

In what follows,
we show that the four conditions of Theorem~\ref{thm:Atkinson}
are fulfilled with $\mathcal{X} = \Vol + \Fred$
and $\mathcal{X}_n = \VolSEn + \FredSEn$,
under the assumptions of Theorem~\ref{thm:SE-Sinc-collocation}.
Condition~1 clearly holds.
Condition~2 is well known, as stated by Theorem~\ref{thm:u-C}.
Condition~3 immediately follows from the Arzel\`{a}--Ascoli theorem.
For condition~4, in the following lemma, we show
\begin{equation}
 \|\{(\Vol + \Fred) - (\VolSEn + \FredSEn)\}(\VolSEn + \FredSEn)\|_{\mathcal{L}(C([a,b]),C([a,b]))} \to 0
\quad (N\to\infty),
\label{eq:operator-norm-converge}
\end{equation}
which implies that
there exists a positive integer $N_0$
such that~\eqref{eq:operator-converge-condition} holds for all
$N\geq N_0$.

\begin{lemma}
\label{lem:operator-norm-converge}
Assume that $k_1(z,\cdot)$ and $k_2(z,\cdot)$
belong to $\HC(\SEt(\domD_d))$
uniformly for $z\in\overline{\SEt(\domD_d)}$,
and $k_1(\cdot,w)$ and $k_2(\cdot,w)$
belong to $\HC(\SEt(\domD_d))$
uniformly for $w\in\overline{\SEt(\domD_d)}$.
Let $h$ be selected by the formula~\eqref{eq:h-SE}.
Then,~\eqref{eq:operator-norm-converge} holds.
\end{lemma}
\begin{proof}
Expanding and arranging the terms, we have
\begin{align}
& \|\{(\Vol + \Fred) - (\VolSEn + \FredSEn)\}(\VolSEn + \FredSEn)\|_{\mathcal{L}(C([a,b]),C([a,b]))}\nonumber\\
&\leq \|(\Vol - \VolSEn)\VolSEn\|_{\mathcal{L}(C([a,b]),C([a,b]))}
+ \|(\Vol - \VolSEn)\FredSEn\|_{\mathcal{L}(C([a,b]),C([a,b]))}\nonumber\\
&\quad +\|(\Fred - \FredSEn)\VolSEn\|_{\mathcal{L}(C([a,b]),C([a,b]))}
+\|(\Fred - \FredSEn)\FredSEn\|_{\mathcal{L}(C([a,b]),C([a,b]))}.
\label{eq:arrange-term}
\end{align}
In an existing study~\cite[Lemma~6.5]{okayama15:_theor},
$\|(\Vol - \VolSEn)\VolSEn\|_{\mathcal{L}(C([a,b]),C([a,b]))}\to 0$
is already proved,
and in another existing study~\cite{okayama11:_improv},
$\|(\Fred - \FredSEn)\FredSEn\|_{\mathcal{L}(C([a,b]),C([a,b]))}\to 0$
is also proved.
The remaining terms to be estimated in~\eqref{eq:arrange-term}
are the second and third terms.
For the second term, setting $F_i(t,s)=k_1(t,s)k_2(s,\SEt(ih))$,
we have
\begin{align*}
 (\Vol - \VolSEn)\FredSEn[f](t)
&= h \sum_{i=-N}^N f(\SEt(ih))\SEtDiv(ih)
\left\{
\int_a^t F_i(t,s)\D{s}
 - \sum_{j=-N}^N F_i(t,\SEt(jh))\SEtDiv(jh)J(j,h)(\SEtInv(t))
\right\}.
\end{align*}
Because $F_i(t,\cdot)$ satisfies the assumptions of
Theorem~\ref{thm:SE-Sinc-indefinite}
uniformly for $t\in [a,b]$ and $i=-N,\,\ldots,\,N$, we have
\begin{align*}
\left|(\Vol - \VolSEn)\FredSEn[f](t)\right|
&\leq h \sum_{i=-N}^N |f(\SEt(ih))| \SEtDiv(ih)
 \left|
\int_a^t F_i(t,s)\D{s}
 - \sum_{j=-N}^N F_i(t,\SEt(jh))\SEtDiv(jh)J(j,h)(\SEtInv(t))
\right|\\
&\leq \|f\|_{C([a,b])} \left(h\sum_{i=-N}^N\SEtDiv(ih)\right)
C_0 \E^{-\sqrt{\pi d \alpha N}}
\end{align*}
for some constant $C_0$. Furthermore,
because $h\sum_{i=-N}^N\SEtDiv(ih)$ converges to $(b - a)$ as $N\to \infty$,
we have
\[
 \|(\Vol - \VolSEn)\FredSEn\|_{\mathcal{L}(C([a,b]),C([a,b]))}
\leq C_0 C_1 \E^{-\sqrt{\pi d \alpha N}}
\]
for some constant $C_1$.
This proves that the second term in~\eqref{eq:arrange-term}
tends to $0$ as $N\to\infty$.
Finally, we investigate the third term in~\eqref{eq:arrange-term}.
Setting $G_j(t,s)=k_2(t,s)k_1(s,\SEt(jh))J(j,h)(\SEtInv(s))$,
we have
\[
 (\Fred - \FredSEn)\VolSEn[f](t)
= \sum_{j=-N}^N f(\SEt(jh))\SEtDiv(jh)
\left\{\int_a^b G_j(t,s)\D{s} - h \sum_{i=-N}^N G_j(t,\SEt(ih))\SEtDiv(ih)
\right\}.
\]
Here, we use Corollary~\ref{cor:SE-Sinc-quadrature},
focusing on $G_j(t,\cdot)$.
This includes $J(j,h)$, which diverges exponentially as
\begin{equation}
 |J(j,h)(x+\I y)|\leq \frac{5 h}{\pi}\cdot\frac{\sinh(\pi y/h)}{\pi y/h},
\label{eq:bound-J}
\end{equation}
as shown in an existing study~\cite[Lemma~6.4]{okayama15:_theor}.
Taking it into account,
setting $K_1 = \max_{z,w\in\overline{\SEt(\domD_d)}}|k_1(z,w)|$
and $K_2 = \max_{z,w\in\overline{\SEt(\domD_d)}}|k_2(z,w)|$,
Corollary~\ref{cor:SE-Sinc-quadrature} yields
\begin{align*}
 \left|
\int_a^b G_j(t,s)\D{s} - h \sum_{i=-N}^N G_j(t,\SEt(ih))\SEtDiv(ih)
\right|
&\leq K_2 K_1 \frac{5 h}{\pi}\cdot\frac{\sinh(\pi d/h)}{\pi d/h}
(b - a)^{2\alpha-1} C_{\alpha,d}^{\textSE}\E^{-\sqrt{\pi d \alpha N}}\\
&= \frac{5 K_1 K_2 (b - a)^{2\alpha-1} C_{\alpha,d}^{\textSE}}{\pi^2 d}
 h^2 \left[\sinh(\pi d/h)\E^{-\pi d/h}\right].
\end{align*}
Furthermore, using $\sinh(\pi d/h)\E^{-\pi d/h}\leq 1/2$
and $h\sum_{j=-N}^N\SEtDiv(jh)\leq C_1$ as previously estimated,
we have
\begin{align*}
 \|(\Fred - \FredSEn)\VolSEn\|_{\mathcal{L}(C([a,b]),C([a,b]))}
\leq \frac{5 K_1 K_2 (b - a)^{2\alpha-1} C_{\alpha,d}^{\textSE} C_1}{2\pi^2 d}
 h,
\end{align*}
which proves that the third term in~\eqref{eq:arrange-term}
tends to $0$ as $N\to\infty$.
This completes the proof.
\end{proof}

Thus, all conditions 1 through 4 in Theorem~\ref{thm:Atkinson}
are fulfilled, and the next theorem is established.

\begin{theorem}
\label{thm:error-SE-Nystroem}
Assume that all the assumptions of Lemma~\ref{lem:operator-norm-converge}
are fulfilled.
Then, there exists a positive integer $N_0$ such that
for all $N\geq N_0$,~\eqref{eq:SE-Nystroem-symbol} has a unique
solution $w\in C([a,b])$.
Furthermore, there exists a constant $C$ independent of $N$ such that
for all $N\geq N_0$, the inequality~\eqref{eq:error-SE-Nystroem} holds.
\end{theorem}
\begin{proof}
From Theorem~\ref{thm:Atkinson}, we have
\[
 \|u - w\|_{C([a,b])}
\leq \|(I - \VolSEn - \FredSEn)^{-1}\|_{\mathcal{L}(C([a,b]),C([a,b]))}
\|(\Vol + \Fred)u - (\VolSEn + \FredSEn)u\|_{C([a,b])}.
\]
If we show that $\|\VolSEn + \FredSEn\|_{\mathcal{L}(C([a,b]),C([a,b]))}$
is uniformly bounded,
then we obtain the
uniform boundedness of
$\|(I - \VolSEn - \FredSEn)^{-1}\|_{\mathcal{L}(C([a,b]),C([a,b]))}$
(see~\eqref{InEq:Bound-Inverse-Op}).
Using the following bound~\cite[Lemma~3.6.5]{stenger93:_numer}
\begin{equation}
 \sup_{x\in\mathbb{R}}|J(j,h)(x)|\leq 1.1 h,
 \label{eq:bound-J-real}
\end{equation}
and setting $\tilde{K}_1=\max_{t,s\in[a,b]}|k_1(t,s)|$, we have
\[
 |\VolSEn[f](t)| \leq
1.1\tilde{K}_1 \|f\|_{C([a,b])} h\sum_{j=-N}^N\SEtDiv(jh)
\leq 1.1 \tilde{K}_1 \|f\|_{C([a,b])} C_1
\]
for some constant $C_1$. Similarly,
setting $\tilde{K}_2=\max_{t,s\in[a,b]}|k_2(t,s)|$, we have
\[
 |\FredSEn[f](t)| \leq
\tilde{K}_2 \|f\|_{C([a,b])} h\sum_{j=-N}^N\SEtDiv(jh)
\leq \tilde{K}_2 \|f\|_{C([a,b])} C_1
\]
for some constant $C_1$.
Therefore, it holds that
\[
 \|\VolSEn + \FredSEn\|_{\mathcal{L}(C([a,b]),C([a,b]))}
\leq C_1\left(1.1\tilde{K}_1 + \tilde{K}_2\right),
\]
which is uniformly bounded. Thus, there exists a constant $C_2$
independent of $N$ such that
\[
 \|(I - \VolSEn - \FredSEn)^{-1}\|_{\mathcal{L}(C([a,b]),C([a,b]))}\leq C_2.
\]
Furthermore, because $k_1(t,\cdot)u(\cdot)$ satisfies
the assumptions of Theorem~\ref{thm:SE-Sinc-indefinite}
and $k_2(t,\cdot)u(\cdot)$ satisfies
the assumptions of Corollary~\ref{cor:SE-Sinc-quadrature},
there exist constants $C_3$ and $C_4$ independent of $N$ such that
\begin{align*}
 \|(\Vol + \Fred)u - (\VolSEn + \FredSEn)u\|_{C([a,b])}
&\leq \|\Vol u - \VolSEn u \|_{C([a,b])}
+ \|\Fred u - \FredSEn u \|_{C([a,b])}
\leq C_3\E^{-\sqrt{\pi d \alpha N}}
+ C_4\E^{-\sqrt{\pi d \alpha N}}.
\end{align*}
This completes the proof.
\end{proof}

\subsection{On the third equation~\eqref{eq:SE-collocation-symbol}}
\label{sec:SE-collocation}

Following Atkinson~\cite[Sect.~4.3]{atkinson97:_numer_solut},
we easily see that $v = \vSEn$ as stated below.
The proof is omitted because its way is fairly standard.

\begin{proposition}
\label{prop:SE-collocation-system}
If~\eqref{eq:SE-collocation-symbol} has a unique solution $v\in C([a,b])$,
then~\eqref{eq:linear-eq-SE-Sinc} is uniquely solvable,
and vice versa. Furthermore, $v = \vSEn$ holds.
\end{proposition}

We also see $v = \ProjSE w$ as shown below.

\begin{proposition}
\label{prop:SE-Nystroem-collocation}
The following two statements are equivalent:
\begin{enumerate}
 \item[{\rm(A)}] Equation~\eqref{eq:SE-Nystroem-symbol}
has a unique solution $w\in C([a,b])$.
 \item[{\rm(B)}] Equation~\eqref{eq:SE-collocation-symbol}
has a unique solution $v\in C([a,b])$.
\end{enumerate}
Furthermore, $v = \ProjSE w$ and $w = g + (\VolSEn+\FredSEn)v$ hold
if the solutions exist.
\end{proposition}
\begin{proof}
Let us first show (A) $\Rightarrow$ (B).
Note that $\VolSEn\ProjSE f = \VolSEn f$ and
$\FredSEn\ProjSE f = \FredSEn f$ hold,
because of the interpolation property: $\ProjSE[f](\SEt(ih)) = f(\SEt(ih))$.
Applying $\ProjSE$ to both sides of~\eqref{eq:SE-Nystroem-symbol},
we have
\[
 \ProjSE w = \ProjSE( g + \VolSEn w + \FredSEn w)
= \ProjSE (g + \VolSEn\ProjSE w + \FredSEn\ProjSE w),
\]
which implies that~\eqref{eq:SE-collocation-symbol} has a solution
$v = \ProjSE w \in C([a,b])$.
Next, we show the uniqueness. Suppose that~\eqref{eq:SE-collocation-symbol}
has another solution $\tilde{v}\in C([a,b])$,
and define a function $\tilde{w}$ as
$\tilde{w} = g + \VolSEn \tilde{v} + \FredSEn\tilde{v}$.
Because $\tilde{v}$ is a solution of~\eqref{eq:SE-collocation-symbol},
we have
\[
 \tilde{v} = \ProjSE(g + \VolSEn\tilde{v} +\FredSEn\tilde{v})
 = \ProjSE\tilde{w},
\]
from which it holds that
\[
 \tilde{w} = g + \VolSEn \tilde{v} + \FredSEn\tilde{v}
= g + \VolSEn \ProjSE\tilde{v} + \FredSEn\ProjSE\tilde{v}
= g + \VolSEn\tilde{w} + \FredSEn\tilde{w}.
\]
This equation implies that $\tilde{w}$ is a solution
of~\eqref{eq:SE-Nystroem-symbol}.
Because the solution of~\eqref{eq:SE-Nystroem-symbol} is unique,
$w = \tilde{w}$ holds, from which we have
$\ProjSE w = \ProjSE\tilde{w}$.
Thus, $v = \tilde{v}$ holds, which shows (B).

The above argument is reversible, which proves (B) $\Rightarrow$ (A).
Furthermore, in view of the above proof,
we observe that $v = \ProjSE w$ and $w = g + (\VolSEn + \FredSEn)v$.
This completes the proof.
\end{proof}

Summing up the above results,
for the proof of Theorem~\ref{thm:SE-Sinc-collocation},
we have shown the following things:
\begin{enumerate}
 \item Equation~\eqref{eq:VFIE} has a unique solution $u\in C([a,b])$
(Theorem~\ref{thm:u-C}).
 \item Equation~\eqref{eq:SE-Nystroem-symbol} has a unique
solution $w\in C([a,b])$ for all sufficiently large $N$,
and the error is estimated as~\eqref{eq:error-SE-Nystroem}
(Theorem~\ref{thm:error-SE-Nystroem}).
 \item Equation~\eqref{eq:linear-eq-SE-Sinc} has a unique
solution $\mathbd{c}\in \mathbb{R}^n$
(the coefficient matrix $(I_n - V_n^{\textSE} - K_n^{\textSE})$
is invertible)
if and only if equation~\eqref{eq:SE-collocation-symbol}
has a unique solution $v\in C([a,b])$,
and $v = \vSEn$ holds (Proposition~\ref{prop:SE-collocation-system}).
 \item Equation~\eqref{eq:SE-collocation-symbol} has a unique
solution $v\in C([a,b])$ for all sufficiently large $N$,
and $v = \ProjSE w$ holds (Proposition~\ref{prop:SE-Nystroem-collocation}).
\end{enumerate}
Then, we obtain the desired error bound by
estimating~\eqref{eq:final-error-SE},
following the guide in Sect.~\ref{sec:sketch-SE}.
This completes the proof of Theorem~\ref{thm:SE-Sinc-collocation}.

\subsection{Sketch of the proof of Theorem~\ref{thm:DE-Sinc-collocation}}
\label{sec:sketch-DE}

In the case of Theorem~\ref{thm:DE-Sinc-collocation} as well,
we consider the following three equations:
\begin{align}
 (I - \Vol - \Fred) u &= g, \nonumber\\
 (I - \VolDEn - \FredDEn) w &= g, \label{eq:DE-Nystroem-symbol}\\
 (I - \ProjDE\VolDEn - \ProjDE\FredDEn) v &= \ProjDE g.
 \label{eq:DE-collocation-symbol}
\end{align}
By the analysis in Sect.~\ref{sec:regularity},
we already know that the first equation has
a unique solution $u\in\MC_{\alpha}(\DEt(\domD_d))$.
In Sect.~\ref{sec:DE-Nystroem},
we analyze the second equation,
and show that~\eqref{eq:DE-Nystroem-symbol} has
a unique solution $w\in C([a,b])$, and there exists a
constant $C$ independent of $N$ such that
\begin{equation}
 \|u - w\|_{C([a,b])}\leq C \frac{\log(2 d N/\alpha)}{N}
 \E^{-\pi d N/\log(2 d N/\alpha)}.
\label{eq:error-DE-Nystroem}
\end{equation}
In Sect.~\ref{sec:DE-collocation},
we analyze the third equation,
and show that~\eqref{eq:DE-collocation-symbol} has
a unique solution $v \in C([a,b])$,
which means that the system~\eqref{eq:linear-eq-DE-Sinc} is
uniquely solvable, i.e.,
the coefficient matrix $(I_n - V_n^{\textDE} - K_n^{\textDE})$
is invertible.
Furthermore, we show that $v = \vDEn = \ProjDE w$ holds.
Thus, we have
\begin{align}
 \|u - \vDEn\|_{C([a,b])}
\leq \|u - \ProjDE u\|_{C([a,b])}
 + \|\ProjDE u - \ProjDE w\|_{C([a,b])}
\leq \|u - \ProjDE u\|_{C([a,b])}
 + \|\ProjDE\|_{\mathcal{L}(C([a,b]),C([a,b]))} \| u - w\|_{C([a,b])}.
\label{eq:final-error-DE}
\end{align}
Because $u\in\MC_{\alpha}(\DEt(\domD_d))$,
we can use Theorem~\ref{thm:DE-Sinc-general} for the first term.
For the second term, we can use~\eqref{eq:error-DE-Nystroem}
and the following lemma.

\begin{lemma}[Okayama~{\cite[Lemma~7.5]{OkayamaFredholm}}]
Let $\ProjDE$ be defined by~\eqref{eq:ProjDE}.
Then, there exists a constant $C$ independent of $N$ such that
\[
 \|\ProjDE\|_{\mathcal{L}(C([a,b]),C([a,b]))}
\leq C\log(N+1).
\]
\end{lemma}

Thus, we obtain the desired error bound
of Theorem~\ref{thm:DE-Sinc-collocation}.

\subsection{On the second equation~\eqref{eq:DE-Nystroem-symbol}}
\label{sec:DE-Nystroem}

Here, we show that~\eqref{eq:DE-Nystroem-symbol} has
a unique solution $w\in C([a,b])$, and estimate the term
$\|u - w\|_{C([a,b])}$.
For this purpose,
we show that the four conditions of Theorem~\ref{thm:Atkinson}
are fulfilled with $\mathcal{X} = \Vol + \Fred$
and $\mathcal{X}_n = \VolDEn + \FredDEn$,
under the assumptions of Theorem~\ref{thm:DE-Sinc-collocation}.
In this case as well,
conditions 1 through 3 can be shown easily.
For condition~4, in the following lemma, we show
\begin{equation}
 \|\{(\Vol + \Fred) - (\VolDEn + \FredDEn)\}(\VolDEn + \FredDEn)\|_{\mathcal{L}(C([a,b]),C([a,b]))} \to 0
\quad (N\to\infty),
\label{eq:operator-norm-converge-DE}
\end{equation}
which implies that
there exists a positive integer $N_0$
such that~\eqref{eq:operator-converge-condition} holds for all
$N\geq N_0$.

\begin{lemma}
\label{lem:operator-norm-converge-DE}
Assume that $k_1(z,\cdot)$ and $k_2(z,\cdot)$
belong to $\HC(\DEt(\domD_d))$
uniformly for $z\in\overline{\DEt(\domD_d)}$,
and $k_1(\cdot,w)$ and $k_2(\cdot,w)$
belong to $\HC(\DEt(\domD_d))$
uniformly for $w\in\overline{\DEt(\domD_d)}$.
Let $h$ be selected by the formula~\eqref{eq:h-DE}.
Then,~\eqref{eq:operator-norm-converge-DE} holds.
\end{lemma}
\begin{proof}
Expanding and arranging the terms, we have
\begin{align}
& \|\{(\Vol + \Fred) - (\VolDEn + \FredDEn)\}(\VolDEn + \FredDEn)\|_{\mathcal{L}(C([a,b]),C([a,b]))}\nonumber\\
&\leq \|(\Vol - \VolDEn)\VolDEn\|_{\mathcal{L}(C([a,b]),C([a,b]))}
+ \|(\Vol - \VolDEn)\FredDEn\|_{\mathcal{L}(C([a,b]),C([a,b]))}\nonumber\\
&\quad +\|(\Fred - \FredDEn)\VolDEn\|_{\mathcal{L}(C([a,b]),C([a,b]))}
+\|(\Fred - \FredDEn)\FredDEn\|_{\mathcal{L}(C([a,b]),C([a,b]))}.
\label{eq:arrange-term-DE}
\end{align}
In an existing study~\cite[Lemma~6.9]{okayama15:_theor},
$\|(\Vol - \VolDEn)\VolDEn\|_{\mathcal{L}(C([a,b]),C([a,b]))}\to 0$
is already proved,
and in another existing study~\cite{okayama11:_improv},
$\|(\Fred - \FredDEn)\FredDEn\|_{\mathcal{L}(C([a,b]),C([a,b]))}\to 0$
is also proved.
The remaining terms to be estimated in~\eqref{eq:arrange-term-DE}
are the second and third terms.
For the second term, setting $F_i(t,s)=k_1(t,s)k_2(s,\DEt(ih))$,
we have
\begin{align*}
 (\Vol - \VolDEn)\FredDEn[f](t)
&= h \sum_{i=-N}^N f(\DEt(ih))\DEtDiv(ih)
\left\{
\int_a^t F_i(t,s)\D{s}
 - \sum_{j=-N}^N F_i(t,\DEt(jh))\DEtDiv(jh)J(j,h)(\DEtInv(t))
\right\}.
\end{align*}
Because $F_i(t,\cdot)$ satisfies the assumptions of
Theorem~\ref{thm:DE-Sinc-indefinite}
uniformly for $t\in [a,b]$ and $i=-N,\,\ldots,\,N$, we have
\begin{align*}
\left|(\Vol - \VolDEn)\FredDEn[f](t)\right|
&\leq h \sum_{i=-N}^N |f(\DEt(ih))| \DEtDiv(ih)
 \left|
\int_a^t F_i(t,s)\D{s}
 - \sum_{j=-N}^N F_i(t,\DEt(jh))\DEtDiv(jh)J(j,h)(\DEtInv(t))
\right|\\
&\leq \|f\|_{C([a,b])} \left(h\sum_{i=-N}^N\DEtDiv(ih)\right)
C_0 \frac{\log(2 d N/\alpha)}{N}\E^{-\pi d N/\log(2 d N/\alpha)}
\end{align*}
for some constant $C_0$. Furthermore,
because $h\sum_{i=-N}^N\DEtDiv(ih)$ converges to $(b - a)$ as $N\to \infty$,
we have
\[
 \|(\Vol - \VolDEn)\FredDEn\|_{\mathcal{L}(C([a,b]),C([a,b]))}
\leq C_0 C_1 \frac{\log(2 d N/\alpha)}{N}\E^{-\pi d N/\log(2 d N/\alpha)}
\]
for some constant $C_1$.
This proves that the second term in~\eqref{eq:arrange-term-DE}
tends to $0$ as $N\to\infty$.
Finally, we investigate the third term in~\eqref{eq:arrange-term-DE}.
Setting $G_j(t,s)=k_2(t,s)k_1(s,\DEt(jh))J(j,h)(\DEtInv(s))$,
we have
\[
 (\Fred - \FredDEn)\VolDEn[f](t)
= \sum_{j=-N}^N f(\DEt(jh))\DEtDiv(jh)
\left\{\int_a^b G_j(t,s)\D{s} - h \sum_{i=-N}^N G_j(t,\DEt(ih))\DEtDiv(ih)
\right\}.
\]
We use Corollary~\ref{cor:DE-Sinc-quadrature}
focusing on $G_j(t,\cdot)$.
Taking~\eqref{eq:bound-J} into account,
setting $K_1 = \max_{z,w\in\overline{\DEt(\domD_d)}}|k_1(z,w)|$
and $K_2 = \max_{z,w\in\overline{\DEt(\domD_d)}}|k_2(z,w)|$,
Corollary~\ref{cor:DE-Sinc-quadrature} yields
\begin{align*}
 \left|
\int_a^b G_j(t,s)\D{s} - h \sum_{i=-N}^N G_j(t,\DEt(ih))\DEtDiv(ih)
\right|
&\leq K_2 K_1 \frac{5 h}{\pi}\cdot\frac{\sinh(\pi d/h)}{\pi d/h}
(b - a)^{2\alpha-1} C_{\alpha,d}^{\textDE}
 \E^{-2\pi d N/\log(2 d N/\alpha)}\\
&= \frac{5 K_1 K_2 (b - a)^{2\alpha-1} C_{\alpha,d}^{\textDE}}{\pi^2 d}
 h^2 \left[\sinh(\pi d/h)\E^{-2\pi d/h}\right].
\end{align*}
Furthermore, using $\sinh(\pi d/h)\E^{-2\pi d/h}\leq \E^{-\pi d /h}/2$
and $h\sum_{j=-N}^N\DEtDiv(jh)\leq C_1$ as previously estimated,
we have
\begin{align*}
 \|(\Fred - \FredDEn)\VolDEn\|_{\mathcal{L}(C([a,b]),C([a,b]))}
\leq \frac{5 K_1 K_2 (b - a)^{2\alpha-1} C_{\alpha,d}^{\textDE} C_1}{2\pi^2 d}
 h \E^{-\pi d/h},
\end{align*}
which proves that the third term in~\eqref{eq:arrange-term-DE}
tends to $0$ as $N\to\infty$.
This completes the proof.
\end{proof}

Thus, all conditions 1 through 4 in Theorem~\ref{thm:Atkinson}
are fulfilled, and the next theorem is established.

\begin{theorem}
\label{thm:error-DE-Nystroem}
Assume that all the assumptions of Lemma~\ref{lem:operator-norm-converge-DE}
are fulfilled.
Then, there exists a positive integer $N_0$ such that
for all $N\geq N_0$,~\eqref{eq:DE-Nystroem-symbol} has a unique
solution $w\in C([a,b])$.
Furthermore, there exists a constant $C$ independent of $N$ such that
for all $N\geq N_0$, the inequality~\eqref{eq:error-DE-Nystroem} holds.
\end{theorem}
\begin{proof}
From Theorem~\ref{thm:Atkinson}, we have
\[
 \|u - w\|_{C([a,b])}
\leq \|(I - \VolDEn - \FredDEn)^{-1}\|_{\mathcal{L}(C([a,b]),C([a,b]))}
\|(\Vol + \Fred)u - (\VolDEn + \FredDEn)u\|_{C([a,b])}.
\]
If we show that $\|\VolDEn + \FredDEn\|_{\mathcal{L}(C([a,b]),C([a,b]))}$
is uniformly bounded,
then we obtain the
uniform boundedness of
$\|(I - \VolDEn - \FredDEn)^{-1}\|_{\mathcal{L}(C([a,b]),C([a,b]))}$
(see~\eqref{InEq:Bound-Inverse-Op}).
Using the bound~\eqref{eq:bound-J-real} and
setting $\tilde{K}_1=\max_{t,s\in[a,b]}|k_1(t,s)|$, we have
\[
 |\VolDEn[f](t)| \leq
1.1\tilde{K}_1 \|f\|_{C([a,b])} h\sum_{j=-N}^N\DEtDiv(jh)
\leq 1.1 \tilde{K}_1 \|f\|_{C([a,b])} C_1
\]
for some constant $C_1$. Similarly,
setting $\tilde{K}_2=\max_{t,s\in[a,b]}|k_2(t,s)|$, we have
\[
 |\FredDEn[f](t)| \leq
\tilde{K}_2 \|f\|_{C([a,b])} h\sum_{j=-N}^N\DEtDiv(jh)
\leq \tilde{K}_2 \|f\|_{C([a,b])} C_1
\]
for some constant $C_1$.
Therefore, it holds that
\[
 \|\VolDEn + \FredDEn\|_{\mathcal{L}(C([a,b]),C([a,b]))}
\leq C_1\left(1.1\tilde{K}_1 + \tilde{K}_2\right),
\]
which is uniformly bounded. Thus, there exists a constant $C_2$
independent of $N$ such that
\[
 \|(I - \VolDEn - \FredDEn)^{-1}\|_{\mathcal{L}(C([a,b]),C([a,b]))}\leq C_2.
\]
Furthermore, because $k_1(t,\cdot)u(\cdot)$ satisfies
the assumptions of Theorem~\ref{thm:DE-Sinc-indefinite}
and $k_2(t,\cdot)u(\cdot)$ satisfies
the assumptions of Corollary~\ref{cor:DE-Sinc-quadrature},
there exist constants $C_3$ and $C_4$ independent of $N$ such that
\begin{align*}
 \|(\Vol + \Fred)u - (\VolDEn + \FredDEn)u\|_{C([a,b])}
&\leq \|\Vol u - \VolDEn u \|_{C([a,b])}
+ \|\Fred u - \FredDEn u \|_{C([a,b])}\\
&\leq C_3\frac{\log(2 d N/\alpha)}{N}\E^{-\pi d N/\log(2 d N/\alpha)}
+ C_4\E^{-2 \pi d N/\log(2 d N/\alpha)}.
\end{align*}
This completes the proof.
\end{proof}

\subsection{On the third equation~\eqref{eq:DE-collocation-symbol}}
\label{sec:DE-collocation}

Following Atkinson~\cite[Sect.~4.3]{atkinson97:_numer_solut},
we easily see that $v = \vDEn$ as stated below.
The proof is omitted because its way is fairly standard.

\begin{proposition}
\label{prop:DE-collocation-system}
If~\eqref{eq:DE-collocation-symbol} has a unique solution $v\in C([a,b])$,
then~\eqref{eq:linear-eq-DE-Sinc} is uniquely solvable,
and vice versa. Furthermore, $v = \vDEn$ holds.
\end{proposition}

We also see $v = \ProjDE w$ as shown below.
The proof is omitted because it proceeds in exactly the same way
as Proposition~\ref{prop:SE-Nystroem-collocation}.

\begin{proposition}
\label{prop:DE-Nystroem-collocation}
The following two statements are equivalent:
\begin{enumerate}
 \item[{\rm(A)}] Equation~\eqref{eq:DE-Nystroem-symbol}
has a unique solution $w\in C([a,b])$.
 \item[{\rm(B)}] Equation~\eqref{eq:DE-collocation-symbol}
has a unique solution $v\in C([a,b])$.
\end{enumerate}
Furthermore, $v = \ProjDE w$ and $w = g + (\VolDEn+\FredDEn)v$ hold
if the solutions exist.
\end{proposition}

Summing up the above results,
for the proof of Theorem~\ref{thm:DE-Sinc-collocation},
we have shown the following things:
\begin{enumerate}
 \item Equation~\eqref{eq:VFIE} has a unique solution $u\in C([a,b])$
(Theorem~\ref{thm:u-C}).
 \item Equation~\eqref{eq:DE-Nystroem-symbol} has a unique
solution $w\in C([a,b])$ for all sufficiently large $N$,
and the error is estimated as~\eqref{eq:error-DE-Nystroem}
(Theorem~\ref{thm:error-DE-Nystroem}).
 \item Equation~\eqref{eq:linear-eq-DE-Sinc} has a unique
solution $\mathbd{c}\in \mathbb{R}^n$
(the coefficient matrix $(I_n - V_n^{\textDE} - K_n^{\textDE})$
is invertible)
if and only if equation~\eqref{eq:DE-collocation-symbol}
has a unique solution $v\in C([a,b])$,
and $v = \vDEn$ holds (Proposition~\ref{prop:DE-collocation-system}).
 \item Equation~\eqref{eq:DE-collocation-symbol} has a unique
solution $v\in C([a,b])$ for all sufficiently large $N$,
and $v = \ProjDE w$ holds (Proposition~\ref{prop:DE-Nystroem-collocation}).
\end{enumerate}
Then, we obtain the desired error bound by
estimating~\eqref{eq:final-error-DE},
following the guide in Sect.~\ref{sec:sketch-DE}.
This completes the proof of Theorem~\ref{thm:DE-Sinc-collocation}.


\bibliography{VolterraFredholm}

\end{document}